\documentclass[12pt]{article} 
\usepackage{amsfonts,amsmath,latexsym,amssymb,mathrsfs,amsthm, url,comment}
\usepackage{slashbox}
\usepackage{caption}

\let\OLDthebibliography\thebibliography
\renewcommand\thebibliography[1]{
  \OLDthebibliography{#1}
  \setlength{\parskip}{3pt}
  \setlength{\itemsep}{0pt plus 0.3ex}
}

\def\numberlikeadb{\global\def\theequation{\thesection.\arabic{equation}}}
\numberlikeadb
\newtheorem{theorem}{Theorem}[section]
\newtheorem{lemma}[theorem]{Lemma}

\newtheorem{proposition}[theorem]{Proposition}
\newtheorem{remark}[theorem]{Remark}

\newtheorem{assumption}[theorem]{Assumption}

\usepackage{color}

\definecolor{orange}{rgb}{1,0.5,0}

\usepackage{lscape}
\usepackage{caption}
\usepackage{multirow}
\begin{document}

\title{An iterative technique for bounding derivatives of solutions of Stein equations}
\author{Christian D\"{o}bler\footnote{Universit\'{e} du Luxembourg, Unit\'{e} de Recherche en Math\'{e}matiques, Maison du Nombre, 4, Avenue de la Fonte, L-4364 Esch-sur-Alzette},\, Robert E. Gaunt\footnote{School of Mathematics, The University of Manchester, Manchester M13 9PL, UK}\, and
 Sebastian Vollmer\footnote{Mathematics Institute, Zeeman Building, University of Warwick, Coventry, CV4 7AL, UK}  
}

\date{\today} 
\maketitle

\vspace{-10mm}

\begin{abstract}We introduce a simple iterative technique for bounding derivatives of solutions of Stein equations $Lf=h-\mathbb{E}h(Z)$, where $L$ is a linear differential operator and $Z$ is the limit random variable.  Given bounds on just the solutions or certain lower order derivatives of the solution, the technique allows one to deduce bounds for derivatives of any order, in terms of supremum norms of derivatives of the test function $h$.  This approach can be readily applied to many Stein equations from the literature.  We consider a number of applications; in particular, we derive new bounds for derivatives of any order of the solution of the general variance-gamma Stein equation.  Finally, we present a connection between Stein equations and Poisson equations, from which we first recognised the importance of the iterative technique to Stein's method.
\end{abstract}


\noindent{{{\bf{AMS 2010 Subject Classification:}}} 60F05; 35A24

\section{Introduction} Stein's method is a powerful technique for assessing the distributional distance between a random variable of interest $W$ and a limit random variable $Z$.  Stein \cite{stein} originally developed the method for normal approximation, and the method was adapted to Poisson approximation by Chen \cite{chen 0}. 

\subsection{Outline of Stein's method} 

For a general target random variable $Z$, Stein's method involves three steps; see \cite{reinert 0}.  In the first step, a suitable characterisation of the target distribution is obtained; namely, a Stein operator $L$ such that the random variable $X$ is equal in law to $Z$ if and only if
\begin{equation}\label{chareqn}\mathbb{E}[Lf(X)]=0
\end{equation}
for all functions $f$ belonging to some measure determining class. For continuous random variables, $L$ is a differential operator; for discrete random variables, $L$ is a difference operator.  Such characterisations have often been derived via Stein's density approach \cite{stein2}, \cite{stein3} or the generator approach of Barbour and G\"{o}tze \cite{barbour2}, \cite{gotze}.  The scope of the density approach has recently been extended by \cite{ley}, and other techniques for obtaining Stein characterisations are discussed in that work. 

The characterisation (\ref{chareqn}) leads to the so-called Stein equation:
\begin{equation}\label{steineqn}Lf(x)=h(x)-\mathbb{E}h(Z),
\end{equation}
where the test function $h$ is real-valued.  The second step of Stein's method, which will be the focus of this paper, concerns the problem of obtaining a solution $f$ to the Stein equation (\ref{steineqn}) and then establishing estimates for $f$ and some of its lower order derivatives (for continuous distributions).

Evaluating (\ref{steineqn}) at a random variable $W$ and taking expectations gives
\begin{equation}\label{steinexpectation}\mathbb{E}h(W)-\mathbb{E}h(Z)=\mathbb{E}[Lf(W)],
\end{equation}
and thus the problem of bounding the quantity $\mathbb{E}h(W)-\mathbb{E}h(Z)$ reduces to solving (\ref{steineqn}) and bounding the right-hand side of (\ref{steinexpectation}).  The third step of Stein's method concerns the problem of bounding the expectation $\mathbb{E}[Lf(W)]$.  For continuous limit distributions, such bounds are usually obtained via Taylor expansions and coupling techniques.   

For many classical distributions, the problem of obtaining the first two necessary ingredients is relatively tractable.  As a result, over the years, Stein's method has been adapted to many standard distributions, including the beta \cite{dobler beta}, \cite{goldstein4}, gamma \cite{gaunt chi square}, \cite{luk}, exponential \cite{cfr11}, \cite{fulmanross13}, \cite{pekoz1}, Laplace \cite{pike} and, more generally, the class of variance-gamma distributions \cite{gaunt vg}.  The method has also been adapted to distributions arising from specific problems, such as preferential attachment graphs \cite{pekoz}, the Curie-Weiss model \cite{cs11} and statistical mechanics \cite{el10}, \cite{em14}.  For a comprehensive overview of the literature, see \cite{ley}.

\subsection{Estimates for solutions of Stein equations}

The Stein equations of many classical distributions, such as the normal, beta and gamma, are linear first order ODEs with simple coefficients.  As a result, the problem of solving the Stein equation and bounding the derivatives of the solution is reasonably tractable.

In fact, for Stein characterisations arising from the generator approach, the Stein operator $L$ associated with the distribution $\mu$ can be viewed as the generator of a Markov process with stationary distribution $\mu$.  For example, the classical standard normal Stein operator $Lf(x)=f'(x)-xf(x)$ is the generator of an Ornstein-Uhlenbeck process (if we take $f=g'$).  In such cases, by standard theory of Markov processes (\cite{kurtz}, Proposition 1.5), the solution to the Stein equation is $f=g'$, where
\begin{equation}\label{gensoln} g(x)=-\int_0^{\infty}\{P_th(x)-\mathbb{E}h(Z)\}\,\mathrm{d}t,
\end{equation}
if the integral exists.  Here, $P_t$, defined by $P_tf(x)=\mathbb{E}_xf(X_t)$, is the transition semigroup operator.  If it is permissible  to differentiate under the integral sign, then
\begin{equation*} f^{(k-1)}(x)=-\int_0^{\infty}(P_th)^{(k)}(x)\,\mathrm{d}t,\quad k\geq1.
\end{equation*}
In cases where a formula is available for $(P_th)^{(k)}(x)$, it is often possible to easily deduce uniform bounds for the derivatives of $f$ in terms of supremum norms of derivatives of $h$.  This approach has yielded uniform bounds for derivatives of any order of the solutions of the multivariate normal \cite{barbour2}, \cite{goldstein1} and gamma \cite{luk} Stein equations.

However, for many Stein equations in the literature, this approach cannot be applied or is not tractable.  This can either be because the Stein operator cannot be viewed as the generator of a Markov diffusion process (for example, the Stein operator for the product of three of more independent standard normals \cite{gaunt pn}) or because a simple formula is not available for $(P_th)^{(k)}(x)$ (for example, the variance-gamma Stein operator; see Section 3.3.2 of \cite{gaunt thesis}).  In such cases, estimates for the solution and its derivatives are usually obtained via analytic methods, by which formulas for the solution and its derivatives are derived, and these quantities are then bounded through technical calculations.  This is often a tedious and lengthy process, which usually only yields bounds on lower order derivatives. 

In recent years, Stein's method has been adapted to a number of distributions for which the Stein operator is a higher order linear differential operator.  Second order operators have been obtained for the Laplace \cite{pike}, variance-gamma \cite{gaunt vg} and PRR (Pek\"oz-R\"ollin-Ross) distributions \cite{pekoz}, whilst operators of order greater than two have been for the distribution of mixed products of independent normal, beta and gamma random variables \cite{gaunt pn}, \cite{gaunt ngb}, \cite{gms16}, and for the distribution of a general linear combination of independent centered gamma random variables \cite{aaps16}.

These higher order Stein operators have simple coefficients, so bounding the quantity $\mathbb{E}[Lf(W)]$ via various coupling techniques is often quite tractable.  However, solving the Stein equation and obtaining estimates for the derivatives of the solution is often a challenging problem.  Indeed, to date, there do not exist uniform bounds for the derivatives of the solution of the general variance-gamma Stein equation.

This paper is to some extent motivated by this challenge.  In this paper, we introduce a simple iterative technique for bounding derivatives of solutions of Stein equations.  We introduce the technique in Section 2 and consider applications of it in Section 3.  Given a bound on just the solution or some of its lower order derivatives, the iterative procedure allows one to deduce bounds for derivatives of any order in terms of supremum norms of the derivatives of $h$.  The technique allows us to obtain bounds for derivatives of any order of the solutions of many of the Stein equations from the literature.  Such bounds have previously only been obtained for the solutions of the Stein equations for the normal and multivariate normal \cite{daly}, \cite{gaunt rate}, \cite{goldstein1}, \cite{lefevre}, gamma \cite{gaunt chi square}, \cite{luk} and beta distributions \cite{dobler beta}.  In this paper, we restrict our attention to Stein equations of continuous distributions, but it should be noted that our iterative technique can be applied analogously to solutions of Stein equations for discrete distributions.  However, one can obtain bounds on higher order forward differences of the solution trivially via the triangle inequality given either a bound on the solution or its first forward difference, without needing to apply our iterative technique.

The power of the technique is particularly well demonstrated if we consider the variance-gamma Stein equation.  For this Stein equation, our iterative procedure reduces the problem of bounding derivatives of any order of the solution to just bounding the solution and its first derivative.  This is a tractable problem, and we obtain uniform bounds for these quantities in Section 3.  On the other hand, bounding even the second and third derivatives via analytic calculations may require very involved calculations; see Section 3 and Appendix D of \cite{gaunt thesis}. 

When the Stein operator $L$ is a second order (possibly degenerate)
elliptic differential operator, then the Stein equation coincides with
the Poisson equation.  It was in the context of the Poisson equation that we first recognised the importance of the iterative technique to Stein's method.  Indeed, in Section 4, we consider the connection between Stein equations and Poisson equations, and consider an application of the iterative procedure to Poisson equations in Section 4.2 

\subsection{Outline of the paper}

In Section 2, we introduce a general iterative technique for bounding derivatives of any order of solutions of Stein equations.  In Lemmas \ref{lemma2.1}--\ref{lemma2.3}, we present general bounds for solutions of certain classes of Stein equations.  These lemmas cover the normal, gamma, beta, Student's $t$, inverse-gamma, PRR and variance-gamma Stein equations.  In Section 3, we apply these lemmas to obtain estimates for $n$-th order derivatives of the solutions to these Stein equations.  Many of these bounds are new.  The technique also applies equally well to many Stein equations not covered by Lemmas \ref{lemma2.1}--\ref{lemma2.3}.  We demonstrate this by considering the application of the technique to the Stein equations for the multivariate normal distribution, as well as to a distribution with density proportional to $e^{-x^4/12}$.  In Section 4, we consider the connection between Stein equations and Poisson equations.  We also consider the iterative technique in the context of Poisson equations in Section 4.2.  For the sake of brevity, some formulas are given in Section 3 without proof.  Their simple proofs are given in Appendix A.

\section{An iterative approach to bounding solutions of Stein equations}

Consider the ordinary differential equation
\begin{equation}\label{ode1}Lf(x)=h(x),
\end{equation}
where $L$ is a linear differential operator of the form
\begin{equation*}Lf(x)=\sum_{j=0}^ma_j(x)f^{(j)}(x).
\end{equation*}
In Stein's method, one frequently encounters linear ODEs of the type (\ref{ode1}). Indeed, usually one is interested in some distribution $\mu$, which is characterised by $L$, meaning that 
\begin{equation*}
 \mathbb{E}[Lf(X)]=0,
\end{equation*}
for all sufficiently smooth functions $f$, if and only if $X$ has distribution $\mu$. One is then interested in solving \eqref{ode1} for $f$ in terms of $h$, which will be supposed to be centred with respect to $\mu$, and in bounding $f$ and derivatives of $f$ in terms of $h$ and its derivatives. In this paper, we address the latter problem within the following general set-up: Suppose that we are given a finite or infinite sequence $(\mu_k,L_k)$ of distributions $\mu_k$ and linear differential operators 
$L_k$ of the form 
\begin{equation}\label{formLk}
L_kf(x)=\sum_{j=0}^ma_{k,j}(x)f^{(j)}(x)
\end{equation}
such that $L_k$ is characterizing for $\mu_k$ in the above mentioned sense.  We shall assume that $a_{k+1,m}(x)=a_{k,m}(x)=\cdots=a_{0,m}(x)$.  This mild assumption (which holds for all Stein equations for univariate distributions in the literature) allows us to make the following important assumption.  We assume that the operators $L_k$ are linked by the relation 
\begin{equation}\label{link}
 \frac{\mathrm{d}}{\mathrm{d}x} L_kf(x)=L_{k+1}f'(x)-T_k f(x),
\end{equation}
where $T_k$ is some other linear differential operator. We will assume that all operators $L_k$ have the same order $m$ and, then, $T_k$ will have the form 
\begin{equation}\label{formTk}
 T_kf(x)=\sum_{j=0}^{m}b_{k,j}(x)f^{(j)}(x),
\end{equation}
for some coefficients $b_{k,j}(x)$ (which are in fact linear combinations of derivatives of the coefficients $a_{k,j}(x)$ of $L_k$).  That $T_k$ has order $m$ can be seen because the leading terms of $\frac{\mathrm{d}}{\mathrm{d}x} L_kf(x)$ and $L_{k+1}f'(x)$ are given, respectively, by $a_{k,m}(x)f^{(m+1)}(x)$ and $a_{k+1,m}f^{(m+1)}(x)$, which are equal because we assume that $a_{k,m}(x)=a_{k+1,m}(x)$.  Thus, if $h_k$ is such that 
\begin{equation}\label{cent}
\mu_k(h_k)=0,
\end{equation}
then, writing $f_k=f_{k,h_k}$ for a particular solution to 
\begin{equation}\label{odek}
 L_k f_k(x)=h_k(x),
\end{equation}
we find that $f_{k+1}=f_k'$ solves 
\begin{equation}\label{odek+1}
 L_{k+1}f_{k+1}(x)=h_k'(x)+T_k f_{k}(x).
\end{equation}
Often, given $h_k$ with \eqref{cent}, there is only one solution $f_{k,h_k}$ to \eqref{odek} which has nice boundedness properties. This is due to the fact that every other solution is obtained by adding to $f_{k,h_k}$ 
a solution to the corresponding homogeneous equation
\begin{equation*}
L_k g(x)=0, 
\end{equation*}
whose non-zero solutions (or one of their lower-order derivatives) often explode in a certain sense. In this general framework, we will always assume that there is one such distinguished solution $f_{k,h_k}$ to \eqref{odek}. 
The curious fact is that then, if the operators $L_m$ and $T_m$ are well-chosen,  usually, the right hand side 
\begin{equation*}
 h_{k+1}(x):=h_k'(x)+T_k f_{k,h_k}(x)
\end{equation*}
of \eqref{odek+1} is centred with respect to the measure $\mu_{k+1}$ and that $f_{k,h_k}'=f_{k+1,h_{k+1}}$. Given this, we can apply bounds holding for the solution $f_{k+1,h_{k+1}}$ in order to bound $f_{k,h_k}'$. 
Iterating this observation already gives the theoretical essence of our method. For ease of reference, we now state the technical conditions. 

\begin{assumption}\label{ass1}
\emph{ Assume that, for some $N\in\mathbb{N}\cup\{\infty\}$ and $m\in\mathbb{N}$, we are given a sequence $(\mu_k,L_k,T_k)$, $0\leq k< N+1$, of probability measures $\mu_k$ on $(\mathbb{R},\mathcal{B})$, 
 linear differential operators $L_k$ and $T_k$ of the 
 form \eqref{formLk} and \eqref{formTk}, respectively, such that $L_k$ is characterizing for $\mu_k$ and $\eqref{link}$ holds. Also, suppose that for each $0\leq k<N+1$ and each function $h_k$ such that $\mu_k(h_k)=0$, there 
 is a distinguished solution $f_{k,h_k}$ to the Stein equation \eqref{odek} and that the function $h_{k+1}:=h_k'+T_k f_{k,h_k}$ is integrable and centred with respect to $\mu_{k+1}$. Then, we know from \eqref{link} that 
 $f_{k,h_k}'$ solves $L_{k+1}f=h_{k+1}$ and we further assume that $f_{k,h_k}'=f_{k+1,h_{k+1}}$ is the distinguished solution of this equation.}
\end{assumption}

\begin{remark}\emph{Note that we do not state in general, what the attribute distinguished in this context precisely means. Indeed, for the general theory, it only has to be guaranteed that the choice of $f_{k,h_k}$ is unambiguous. In all our examples in the sequel except for the variance-gamma Stein equation when $0<r<1$ (see Section 3.1.7 for the p.d.f.), the distinguished solution will be the unique bounded solution. In case of the variance-gamma Stein equation when $0<r<1$, there are infinitely many bounded solutions to the Stein equation, for general bounded test function $h$.  However, if $h$ is bounded, then for all parameter values there is a unique solution $f$ to the variance-gamma Stein equation such that $f$ and $f'$ are bounded (see \cite{gaunt thesis}, Lemma 3.14). This then becomes the distinguished solution.}
\end{remark}

Now, we present three general lemmas, which provide general useful bounds under different additional conditions, if the operators $T_k$ have a special and simple form.  As we shall see in Section 3, these lemmas cover many of the Stein equations from the literature.  In the following lemmas, we set the empty product $\prod_{i=0}^{-1}a_i$ to be equal to $1$.

\begin{lemma}\label{lemma2.1} Suppose that Assumption \ref{ass1} holds and that there are constants $\alpha_j$, $0\leq j<N+1$, such that $T_jf=\alpha_j f$ holds for all $0\leq j<N+1$, where $f$ is distinguished solution according to Assumption \ref{ass1}. Also, assume that for some 
$0\leq n<N+1$ a function $h\in C_b^n(\mathbb{R})$ with $\mu_0(h)=0$ is given and write $f$ for $f_{0,h}$.
Then, writing $a_j:=|\sum_{i=0}^j\alpha_i|$, $0\leq j<N+1$:

\begin{enumerate}
\item[(i)] If there exist constants $C_l$, $0\leq l<N+1$, such that $\|f_{l,h_l}\|\leq C_l\|h_l\|$ holds for each $0\leq l<N+1$ and $h_l$ with $\mu_l(h_l)=0$, then we have
\begin{equation*}\|f^{(n)}\|\leq C_n\sum_{j=0}^n\bigg(\prod_{i=j}^{n-1}C_i a_{i}\bigg)\|h^{(j)}\|.
\end{equation*}

\item[(ii)]  If there exist constants $D_l$, $0\leq l<N+1$, such that $\|f_{l,h_l}'\|\leq D_l\|h_l\|$ holds for each $0\leq l<N+1$ and $h_l$ with $\mu_l(h_l)=0$, then we have, for $n=2k+1$,
\begin{equation*}\|f^{(2k+1)}\|\leq D_{2k}\sum_{j=0}^k\bigg(\prod_{i=j}^{k-1}D_{2i} a_{2i+1}\bigg)\|h^{(2j)}\|
\end{equation*}
and, for $n=2k$,
\begin{equation*}\|f^{(2k)}\|\leq D_{2k-1}\sum_{j=1}^k\bigg(\prod_{i=j}^{k-1}D_{2i-1} a_{2i}\bigg)\|h^{(2j-1)}\|+\|f\|\prod_{i=1}^kD_{2i-1} a_{2i-2}.
\end{equation*}
\item[(iii)] If $1\leq n<N+1$ and if there exist constants $E_l$, $0\leq l<N+1$, such that $\|f_{l,h_l}'\|\leq E_l\|h_l'\|$ holds for each $0\leq l<N+1$ and Lipschitz-continuous $h_l$ with $\mu_l(h_l)=0$, then we have
\begin{equation*}\|f^{(n)}\|\leq E_{n-1}\sum_{j=1}^n\bigg(\prod_{i=j}^{n-1}E_{i-1} a_{i-1}\bigg)\|h^{(j)}\|.
\end{equation*}

\end{enumerate}

\end{lemma}

\begin{proof}
From the hypothesis of the lemma we obtain inductively that
\begin{equation*}
 L_{j+1}f^{(j+1)}=h_{j+1}=h_j'+\alpha_j f^{(j)}=\ldots=h^{(j+1)}+\Bigl(\sum_{i=0}^j\alpha_{i}\Bigr) f^{(j)}
\end{equation*}
holds for each $0\leq j<N$. This, in particular implies that 
\begin{equation*}
 \|h_{j+1}\|\leq \|h^{(j+1)}\|+\Bigl|\sum_{i=0}^j\alpha_i\Bigr|\|f^{(j)}\|
\end{equation*}
for each $0\leq j<N$.

(i) The result follows from a simple induction on $n$.  Suppose the result is true for all $l\leq n$.  Then
\begin{align*}\|f^{(n+1)}\|&\leq C_{n+1}\|h^{(n+1)}+\Bigl(\sum_{i=0}^n\alpha_{i}\Bigr) f^{(n)}\| \\
&\leq C_{n+1}\|h^{(n+1)}\|+C_{n+1} a_{n}\|f^{(n)}\| \\
&\leq C_{n+1}\|h^{(n+1)}\|+C_{n+1} a_{n}\cdot C_n\sum_{j=0}^n\bigg(\prod_{i=j}^{n-1}C_i a_{i}\bigg)\|h^{(j)}\| \\
&= C_{n+1}\sum_{j=0}^{n+1}\bigg(\prod_{i=j}^{n}C_i a_{i}\bigg)\|h^{(j)}\|,
\end{align*}  
and so the result has been proved by induction on $n$.

(ii) We prove the result by induction by considering the cases of odd and even $n$ separately.  Firstly, suppose the result holds for all odd $l\leq 2k+1$. Then
\begin{align*}\|f^{(2k+3)}\|&\leq D_{2k+2}\|h^{(2k+2)}+\Bigl(\sum_{i=0}^{2k+1}\alpha_{i}\Bigr) f^{(2k+1)}\| \\
&\leq D_{2k+2}\|h^{(2k+2)}\|+D_{2k+2} a_{2k+1}\|f^{(2k+1)}\| \\
&\leq D_{2k+2}\|h^{(2k+2)}\|+D_{2k+2} a_{2k+1}\cdot D_{2k}\sum_{j=0}^k\bigg(\prod_{i=j}^{k-1}D_{2i} a_{2i+1}\bigg)\|h^{(2j)}\| \\
&= D_{2k+2}\sum_{j=0}^{k+1}\bigg(\prod_{i=j}^{k}D_{2i} a_{2i+1}\bigg)\|h^{(2j)}\|,
\end{align*}
as required.  Now, we suppose the result holds for all even $l\leq 2k$.  Then
\begin{align*}\|f^{(2k+2)}\|&\leq D_{2k+1}\|h^{(2k+1)}+ \Bigl(\sum_{i=0}^{2k}\alpha_{i}\Bigr) f^{(2k)}\| \\
&\leq D_{2k+1}\|h^{(2k+1)}\|+D_{2k+1} a_{2k}\|f^{(2k)}\| \\
&\leq D_{2k+1}\|h^{(2k+1)}\|+D_{2k+1} a_{2k}\cdot D_{2k-1}\sum_{j=1}^k\bigg(\prod_{i=j}^{k-1}D_{2i-1} a_{2i}\bigg)\|h^{(2j-1)}\|\\
&\quad+D_{2k+1} a_{2k-1}\|f\|\prod_{i=1}^{k-1}D_{2i-1} a_{2i-2} \\
&=D_{2k+1}\sum_{j=1}^{k+1}\bigg(\prod_{i=j}^{k}D_{2i-1} a_{2i}\bigg)\|h^{(2j-1)}\|+\|f\|\prod_{i=1}^{k}D_{2i-1} a_{2i-2},
\end{align*}
which completes the proof. The proof of (iii) is very similar to that of (i) and is therefore omitted. 
\end{proof} 

In the following lemma we suppose now that $T_jf=\alpha_{j} f'$ instead of $T_jf=\alpha_{j} f$. The proof is analogous to that of Lemma \ref{lemma2.1} and is omitted.

\begin{lemma}\label{lemma2.2}
Suppose that Assumption \ref{ass1} holds and that there are constants $\alpha_j$, $0\leq j<N+1$, such that $T_jf=\alpha_{j} f'$ holds for all $0\leq j<N+1$. Also, assume that for some 
$0\leq n<N+1$ a function $h\in C_b^n(\mathbb{R})$ with $\mu_0(h)=0$ is given and write $f$ for $f_{0,h}$.
Then, writing $a_j:=|\sum_{i=0}^j \alpha_i|$, $0\leq j<N+1$:

\begin{enumerate}
\item[(i)] If there exist constants $C_l$, $0\leq l<N+1$, such that $\|f_{l,h_l}'\|\leq C_l\|h_l\|$ holds for each $0\leq l<N+1$ and $h_l$ with $\mu_l(h_l)=0$, then we have
\begin{equation*}\|f^{(n)}\|\leq C_{n-1}\sum_{j=0}^{n-1}\bigg(\prod_{i=j}^{n-2}C_i a_{i}\bigg)\|h^{(j)}\|.
\end{equation*}

\item[(ii)] If there exists constants $D_l$, $0\leq l<N+1$, such that $\|f_{l,h_l}''\|\leq D_l\|h_l\|$ holds for each $0\leq l<N+1$ and $h_l$ with $\mu_l(h_l)=0$, then we have, for $1\leq n=2k+1$,
\begin{equation*}
\|f^{(2k+1)}\|\leq D_{2k-1}\sum_{j=1}^k\bigg(\prod_{i=j}^{k-1}D_{2i-1} a_{2i}\bigg)\|h^{(2j-1)}\|+\|f'\|\prod_{i=1}^kD_{2i-1} a_{2i-2}
\end{equation*}
and, for $2\leq n=2k$,
\begin{equation*}\|f^{(2k)}\|\leq D_{2k-2}\sum_{j=1}^{k-1}\bigg(\prod_{i=j}^{k-2}D_{2i} a_{2i+1}\bigg)\|h^{(2j)}\|.
\end{equation*}
\end{enumerate}
\end{lemma}

In the following lemma we generalise the settings of Lemmas \ref{lemma2.1} and \ref{lemma2.2} to suppose that $T_jf=\alpha_{j+1} f+\beta_{j+1}f'$.

\begin{lemma}\label{lemma2.3}
Suppose that Assumption \ref{ass1} holds and that there are constants $\alpha_j$ and $\beta_j$, $1\leq j<N+2$, such that $T_jf=\alpha_{j+1} f+\beta_{j+1}f'$ holds for all $0\leq j<N+1$. Also, assume that for some 
$1\leq m<N+1$ a function $h\in C_b^m(\mathbb{R})$ with $\mu_0(h)=0$ is given and write $f$ for $f_{0,h}$.

If there exist constants $K_j$ and $D_j$, $0\leq j<N+1$, such that $\|f_{j,h_j}\|\leq K_j\|h_j\|$ and $\|f_{j,h_j}'\|\leq D_j\|h_j\|$ hold for each $0\leq j<N+1$ and $h_k$ with $\mu_k(h_k)=0$, then,
with $a_j:=|\sum_{i=0}^j\alpha_i|$ and $b_j:=|\sum_{i=0}^j\beta_i|$, we have that 
\begin{align*}
 \|f^{(m+1)}\|&\leq\sum_{j=0}^{m-1} D_{m-j}\bigg(\sum_{l\in I_j} A_{j,l}^{(m)}\bigg)\|h^{(m-j)}\|+\bigg(D_0\sum_{l\in I_{m}}A_{m,l}^{(m)}+a_1K_0D_1\sum_{l\in I_{m-1}}A_{m-1,l}^{(m)}\bigg)\|h\|.
\end{align*}
Here, for $j\in\{0,1,\dotsc\}$, we let 
\begin{equation*}
 I_j:=\Bigl\{j+1-\Bigl\lfloor\frac{j}{2}\Bigr\rfloor,\dotsc,j+1\Bigr\}\quad\text{such that}\quad |I_j|=1+\Bigl\lfloor\frac{j}{2}\Bigr\rfloor
\end{equation*}
and, for $l\in I_j$, we define the constants 
\begin{equation*}
 A_{j,l}^{(m)}:=\sum_{M\in\mathcal{S}_{j,l}^{(m)}}\prod_{i\in M\setminus\{m-j\}}\big(b_i 1_{\{i-1\in M\}}+a_i 1_{\{i-1\notin M\}}\big)D_i,
\end{equation*}
where $\mathcal{S}_{j,l}^{(m)}$ is the collection of those size $l$ subsets $M$ of $M_j^{(m)}:=\{m-j,m-j+1,\dotsc,m\}$ such that $\{m,m-j\}\subseteq M$ and $M\cap \{i,i+1\}\not=\emptyset$ for all $m-j\leq i\leq m-1$.
\end{lemma}

\begin{proof}
By induction one can see that 
\begin{equation*}
 L_{j+1}f^{(j+1)}=h^{(j+1)}+\Bigl(\sum_{i=1}^{j+1}\alpha_{i}\Bigr)f^{(j)}+\Bigl(\sum_{i=1}^{j+1}\beta_{i}\Bigr)f^{(j+1)}
\end{equation*}
holds for all $0\leq j<N$.
We will occasionally suppress the index $m$ from the notation. Note that for each fixed $1\leq k\leq m$ we have 
\begin{equation*}
 \|f^{(k+1)}\|\leq D_k\|h^{(k)}\|+b_kD_k\|f^{(k)}\|+a_kD_k\|f^{(k-1)}\|.
\end{equation*}
From this, for each $0\leq k\leq m-1$, we can obtain a formula of the form 
\begin{equation}\label{backward}
 \|f^{(m+1)}\|\leq\sum_{j=0}^k C_1^{(m-j)}\|h^{(m-j)}\|+C_2^{(m-k)}\|f^{(m-k)}\|+C_3^{(m-k)}\|f^{(m-k-1)}\|
\end{equation}
with sequences $C_1^{(j)}, C_2^{(j)}, C_3^{(j)}$, $j=1,\dotsc,m$, which satisfy the following recursive relations:
\begin{align}
 C_1^{(k-1)}&=D_{k-1}C_2^{(k)}\label{c1}\\
 C_2^{(k-1)}&=C_3^{(k)}+b_{k-1}D_{k-1}C_2^{(k)}\quad\text{and}\label{c2}\\
 C_3^{(k-1)}&=a_{k-1}D_{k-1}C_2^{(k)}.\label{c3}
\end{align}
The initial values are given by 
\begin{equation}\label{ini}
 C_1^{(m)}=D_m\,,\quad C_2^{(m)}=b_mD_m\quad\text{and}\quad C_3^{(m)}=a_mD_m.
\end{equation}
Choosing $k=m-1$ in \eqref{backward}, we obtain
\begin{align}\label{fmplus1}
 \|f^{(m+1)}\|&\leq\sum_{j=0}^{m-1} C_1^{(m-j)}\|h^{(m-j)}\|+ C_2^{(1)}\|f'\|+C_3^{(1)}\|f\|\notag\\
 &\leq \sum_{j=0}^{m-1} C_1^{(m-j)}\|h^{(m-j)}\|+D_0C_2^{(1)}\|h\|+K_0C_3^{(1)}\|h\|\notag\\
 &=\sum_{j=0}^{m-1} C_1^{(m-j)}\|h^{(m-j)}\|+C_1^{(0)}\|h\|,
\end{align}
where
\[C_1^{(0)}:=D_0C_2^{(1)}+K_0C_3^{(1)}.\]
Thus, it suffices to find explicit expressions for the sequences $C_1^{(m-j)}, C_2^{(m-j)}, C_3^{(m-j)}$, $j=0,1,\dotsc,m-1$. We claim that these are given by 
\begin{align*}
 C_1^{(m-j)}&=\sum_{l\in I_j}D_{m-j}A_ {j,l}=\sum_{l\in I_j}\sum_{M\in\mathcal{S}_{j,l}}D_{m-j}\prod_{i\in M\setminus\{m-j\}}\big(b_i 1_{\{i-1\in M\}}+a_i 1_{\{i-1\notin M\}}\big)D_i,\\
C_2^{(m-j)}&=\sum_{l\in I_{j+1}} A_{j+1,l}=\sum_{l\in I_{j+1}}\sum_{M\in\mathcal{S}_{j+1,l}}\prod_{i\in M\setminus\{m-j-1\}}\big(b_i 1_{\{i-1\in M\}}+a_i 1_{\{i-1\notin M\}}\big)D_i,\\
 C_3^{(m-j)}&=\sum_{l\in I_j}a_{m-j}D_{m-j}A_ {j,l}=\sum_{l\in I_j}\sum_{M\in\mathcal{S}_{j,l}}a_{m-j}D_{m-j}\prod_{i\in M\setminus\{m-j\}}\big(b_i 1_{\{i-1\in M\}}+a_i 1_{\{i-1\notin M\}}\big)D_i.
 \end{align*}
This will be shown by induction on $j=0,1,\dotsc,m-1$. If $j=0$, then the claim is true by virtue of \eqref{ini}. Suppose that it holds for some $0\leq j\leq m-2$. Then, from \eqref{c1} and the induction hypothesis 
for $C_2^{(m-j)}$ we obtain 
\begin{align*}
 C_1^{(m-j-1)}&=D_{m-j-1}C_2^{(m-j)}=\sum_{l\in I_{j+1}} D_{m-j-1}A_{j+1,l},
\end{align*}
as claimed. Similaraly, we obtain the claim for $C_3^{(m-j-1)}$. As to $C_2^{(m-j-1)}$, from \eqref{c2} and the induction hypothesis for $C_2^{(m-j)}$ and for $C_3^{(m-j)}$ we obtain that 
\begin{align}\label{indc21}
 C_2^{(m-j-1)}&=C_3^{(m-j)}+b_{m-j-1}D_{m-j-1}C_2^{(m-j)}\notag\\
 &=\sum_{l\in I_j}a_{m-j}D_{m-j}A_ {j,l}+\sum_{l\in I_{j+1}}b_{m-j-1}D_{m-j-1} A_{j+1,l}.
\end{align}
Suppose that $l\in I_j$. Then, if $M\in\mathcal{S}_{j,l}$, we have $M':=M\cup\{m-j-2\}\in\mathcal{S}_{j+2,l+1}$ and $m-j-1\notin M'$. This implies that 
\begin{align*}\label{indc22}
 a_{m-j}D_{m-j}A_ {j,l}&=\sum_{M\in\mathcal{S}_{j,l}}a_{m-j}D_{m-j}\prod_{i\in M\setminus\{m-j\}}D_i\big(b_i 1_{\{i-1\in M\}}+a_i 1_{\{i-1\notin M\}}\big)\\
 &=\sum_{\substack{M'\in\mathcal{S}_{j+2,l+1}:\\ m-j-1\notin M'}}\prod_{i\in M'\setminus\{m-j-2\}}D_i\big(b_i 1_{\{i-1\in M\}}+a_i 1_{\{i-1\notin M\}}\big)\,,
\end{align*}
because, for $M'\in\mathcal{S}_{j+2,l+1}$, the condition that $m-j-1\notin M'$ necessitates that $m-j\in M'$. Furthermore, if $j$ is even, then it is easy to see that 
\begin{equation*}\label{indc23}
 a_{m-j}D_{m-j}A_ {j,j/2+1}=A_{j+2,j/2+2},
\end{equation*}
because $M'\in\mathcal{S}_{j+2,j/2+2}$ necessarily implies that $m-j-1\notin M'$.

Similarly, if $l\in I_{j+1}$ and $N\in\mathcal{S}_{j+1,l}$, then $N':=N\cup\{m-j-2\}\in\mathcal{S}_{j+2,l+1}$ and $m-j-1\in N\subseteq N'$ and, thus, 
\begin{align*}\label{indc24}
 b_{m-j-1}D_{m-j-1} A_{j+1,l}&=\sum_{N\in\mathcal{S}_{j+1,l}}b_{m-j-1}D_{m-j-1}\prod_{i\in N\setminus\{m-j-1\}}D_i\big(b_i 1_{\{i-1\in N\}}+a_i 1_{\{i-1\notin N\}}\big)\notag\\
 &=\sum_{\substack{N'\in\mathcal{S}_{j+2,l+1}:\\m-j-1\in N'}}\prod_{i\in N'\setminus\{m-j-2\}}D_i\big(b_i 1_{\{i-1\in N'\}}+a_i 1_{\{i-1\notin N'\}}\big).
\end{align*}
Also, if $j$ happens to be odd, then is is easily checked that
\begin{equation}\label{indc25}
 b_{m-j-1}D_{m-j-1}A_{j+1,j+2}=A_{j+2,j+3},
\end{equation}
because for each $j$ the only size $j+1$ subset of $M_j^{(m)}$ is $M_j^{(m)}$ itself. Thus, from \eqref{indc21}-\eqref{indc25} we conclude that, indeed,
\begin{equation*}
 C_2^{(m-j-1)}=\sum_{l\in I_{j+2}}A_{j+2,l},
\end{equation*}
as claimed. This finishes the induction step and \eqref{fmplus1} now yields 
\begin{align*}
 \|f^{(m+1)}\|&\leq\sum_{j=0}^{m-1}C_1^{(m-j)}\|h^{(m-j)}\|+\bigl(D_0C_2^{(1)}+K_0C_3^{(1)}\bigr)\|h\|\\
 &=\sum_{j=0}^{m-1} D_{m-j}\bigg(\sum_{l\in I_j} A_{j,l}^{(m)}\bigg)\|h^{(m-j)}\|+\bigg(D_0\sum_{l\in I_{m}}A_{m,l}^{(m)}+a_1K_0D_1\sum_{l\in I_{m-1}}A_{m-1,l}^{(m)}\bigg)\|h\|,
\end{align*}
completing the proof.

\end{proof}

\section{Bounds for derivatives of solutions of Stein equations via the iterative technique}

In this section, we apply the iterative technique, that was introduced in Section 2, to obtain bounds for derivatives of any order for the solutions of a number of Stein equations from the literature.  The iterative method presented in this paper is particularly suitable for linear Stein equations, whose coefficients are polynomials of order at most one.  In particular, we can directly apply Lemmas \ref{lemma2.1}--\ref{lemma2.3} to obtain estimates for the solutions of the normal, gamma, beta, Student's $t$, PRR and variance-gamma distributions.  

However, the iterative approach also applies equally well to many linear Stein equations not covered by these lemmas.  We illustrate this with an application to a distribution with density proportional  to $e^{-x^4/12}$, which appears as the distributional limit of the magnetization in a Curie-Weiss model.  The iterative method can also be used for Stein equations that are linear PDEs, and we demonstrate this with an application of the iterative technique to the multivariate normal Stein equation. 

\subsection{Estimates for solutions of Stein equations}

Here, we consider the application of the iterative method to a number of Stein equations from the literature.  For each of the distributions considered, we state the Stein equation and its solution, as well as the best existing bounds for the solution and its derivatives.  For each distribution, the iterative technique can be applied to bound derivatives of any order given bounds on just the solution or some of its lower order derivatives.  The iterative technique does not lead to improvements in estimates for derivatives of the solution of the normal and gamma Stein equations, although it does allow for simpler derivations at the expense of slightly larger constants (see Sections 3.2.1 and 3.2.2).  However, we do obtain new bounds for Student's $t$ and the inverse-gamma distributions, the distribution with density proportional to $e^{-x^4/12}$), as well as new bounds for higher order derivatives of the PRR and variance-gamma distributions.  The results are stated without further comment; proofs and further discussions are given in Section 3.2.  It should be understood that the bounds stated below are new, unless a citation is given.

\vspace{3mm}

\noindent{\emph{Notation.}} The double factorial of a positive integer $n$ is given by
\begin{equation*}n!!=\begin{cases} 1\cdot 3\cdot 5\cdot \cdots (n-2)\cdot n, & \text{$n>0$ odd} \\
2\cdot 4\cdot 6 \cdots (n-2)\cdot n, & \text{$n>0$ even}
\end{cases}
\end{equation*}
and we define $(-1)!!=0!!=1$ (\cite{arfken}, p$.$ 547).  The Pochhammer symbol is defined by $(x)_n=x(x+1)\cdots(x+n-1)$ and the Pochhammer $k$-symbol is defined as $(x)_{n,k}=x(x+k)(x+2k)\cdots(x+(n-1)k)$.  As in Section 2, we set the empty product $\prod_{i=0}^{-1}a_i$ to be equal to $1$.  We say that the function $h:S\subseteq\mathbb{R}\rightarrow\mathbb{R}$ belongs to the class $C_b^n(S)$ if $h^{(n-1)}$ exists and is absolutely continuous, with derivatives up to $n$-th order bounded.  It shall also be convenient to let $\tilde{h}(x)=h(x)-\mathbb{E}h(Z)$ and $h^{(0)}\equiv h$.  For a function $f:S\subseteq\mathbb{R}\rightarrow\mathbb{R}$, we shall write $\|f\|=\|f\|_\infty=\sup_{x\in S}|f(x)|$.  Finally, we shall let $D$ denote the usual differential operator and $I$ denote the identity operator

\subsubsection{Normal distribution}

\begin{enumerate}
\item Let $Z\sim N(0,1)$ with density $p(x)=\frac{1}{\sqrt{2\pi}}e^{-x^2/2}$, $x\in\mathbb{R}$.

\item The Stein equation is (see \cite{stein}, \cite{stein2})
\begin{equation}\label{311}Lf(x)=f'(x)-xf(x)=\tilde{h}(x)
\end{equation}
where $\tilde{h}(x)=h(x)-\mathbb{E}h(Z)$, and, for $k\geq1$,
\begin{equation}\label{311n}Lf^{(k)}(x)=h^{(k)}(x)+kf^{(k-1)}(x).
\end{equation}

We can verify (\ref{311n}) by induction on $k$.  Differentiating both sides of (\ref{311n}) gives
\begin{equation}\label{indproof}\frac{\mathrm{d}}{\mathrm{d}x}Lf^{(k)}(x)=f^{(k+2)}(x)-xf^{(k+1)}(x)-f^{(k)}(x)=h^{(k+1)}(x)+kf^{(k)}(x),
\end{equation}
which on rearranging yields
\begin{equation*}Lf^{(k+1)}(x)=h^{(k+1)}(x)+(k+1)f^{(k)}(x),
\end{equation*}
as required.  From (\ref{indproof}) we can read off that $(L_k,T_k)=(L,I)$ for $k\geq0$.  We can therefore bound the derivatives of the solution via Lemma \ref{lemma2.1}, with $a_k=k+1$ for $k\geq0$.

\item The solution to (\ref{311}) is
\begin{equation}\label{311soln}  f(x)=e^{x^2/2}\int_{-\infty}^x\tilde{h}(t)e^{-t^2/2}\,\mathrm{d}t.
\end{equation}

\item The following bounds for the solution were obtained by \cite{stein2}:
\begin{equation}\label{snbound}  \|f\| \leq \sqrt{\frac{\pi}{2}}\|\tilde{h}\|, \quad \|f'\| \leq 2\|\tilde{h}\|, \quad \|f''\| \leq 2\|h'\|.
\end{equation}
These bounds have since been extended to higher order derivatives.  For $h\in C_b^n(\mathbb{R})$, $n\geq0$, the solution satisfies (see \cite{goldstein1},  \cite{gaunt rate}, \cite{daly} respectively)
\begin{eqnarray}\label{snbound1}\|f^{(n)}\|&\leq&\frac{\|h^{(n+1)}\|}{n+1},\\ 
\label{snbound2}\|f^{(n)}\|&\leq&\frac{\Gamma(\frac{n+1}{2})\|h^{(n+1)}\|}{\sqrt{2}\Gamma(\frac{n}{2}+1)}, \\
\label{snbound3}\|f^{(n+2)}\|&\leq& 2\|h^{(n+1)}\|.
\end{eqnarray}

\end{enumerate}

\subsubsection{Gamma distribution}

\begin{enumerate}
\item Let $Z\sim \Gamma(r,\lambda)$, $r,\lambda>0$, with density $p(x)=\frac{\lambda^r}{\Gamma(r)}x^{r-1}e^{-\lambda x}$, $x>0$.

\item The Stein equation is (see \cite{diaconis})
\begin{equation}\label{322}L_{r,\lambda}f(x)=xf'(x)+(r-\lambda x)f(x)=\tilde{h}(x)
\end{equation}
and, for $k\geq1$,
\begin{equation}\label{gamid}L_{r+k,\lambda}f^{(k)}(x)=h^{(k)}(x)+k\lambda f^{(k-1)}(x).
\end{equation}
As was the case for the normal distribution, we can verify (\ref{gamid}) by induction.  This is the case for all the subsequent distributions, and for the sake of brevity we omit the details for all subsequent distributions.  Differentiating both sides of (\ref{322}) gives
\begin{align}\frac{\mathrm{d}}{\mathrm{d}x}L_{r+k,\lambda}f^{(k)}(x)&=xf^{(k+2)}(x)+(r+k-\lambda x)f^{(k+1)}(x)+f^{(k+1)}(x)-\lambda f^{(k)}(x)\nonumber\\
\label{readoff}&=h^{(k+1)}(x)+k\lambda f^{(k)}(x),
\end{align}
which on rearranging gives
\begin{align*}L_{r+k+1,\lambda}f^{(k+1)}(x)=h^{(k+1)}(x)+(k+1)\lambda f^{(k)}(x),
\end{align*}
as required.  From (\ref{readoff}), we can read off that $(L_k,T_k)=(L_{r+k,\lambda},\lambda I)$ for $k\geq0$.   We can therefore bound the derivatives of the solution via Lemma \ref{lemma2.1}, with $a_k=\lambda(k+1)$ for $k\geq0$.

\item The solution to (\ref{322}) is
\begin{equation}\label{322soln}f(x)=\frac{e^{\lambda x}}{x^r}\int_0^{x}t^{r-1}e^{-\lambda t}\tilde{h}(t)\,\mathrm{d}t.
\end{equation}

\item For $h\in C_b^n(\mathbb{R}_+)$, $n\geq0$, the solution satisfies (see \cite{luk},  \cite{gaunt thesis}, \cite{gaunt chi square} respectively)
\begin{eqnarray}\label{chisquarebound1}\|f^{(n-1)}\|&\leq&\frac{1}{n\lambda}\|h^{(n)}\|,  \\
\label{chisquarebound2} \|f^{(n)}\|&\leq&\bigg\{\frac{\sqrt{2\pi}+e^{-1}}{\sqrt{r+n}}+\frac{2}{r+n}\bigg\}\|h^{(n)}\|, \\
\label{chisquarebound3} \|f^{(n)}\|&\leq&\frac{2}{r+n}\big\{3\|h^{(n)}\|+2\lambda\|h^{(n-1)}\|\big\}. 
\end{eqnarray}
A straightforward application of Lemma 1 of \cite{schoutens0} yields
\begin{equation}\label{schgam}\|f\|\leq \frac{e^r\Gamma(r)}{r^r}\|\tilde{h}\|,
\end{equation}
which has a smaller constant than (\ref{chisquarebound2}) when $n=0$.

\item Consider the $\mathrm{Exp}(\lambda)\stackrel{\mathcal{D}}{=}\Gamma(1,\lambda)$ distribution.  For the case $r=1$, some of the estimates for the solution (\ref{322}) can be improved.  For a continuously differentiable test function $h$ with a Lipschitz derivative $h'$, \cite{dobler beta} obtained the bounds 
\begin{equation*}
 \|f\|\leq\frac{1}{\lambda}\|h'\|,\quad\|f'\|\leq\|h'\|,\quad\|f''\|\leq\frac{2\lambda}{3}\|h'\|+\frac{2}{3}\|h''\|,
\end{equation*}
which have better constants than (\ref{chisquarebound2}) and (\ref{chisquarebound3}).

\end{enumerate}

\subsubsection{Beta distribution}

\begin{enumerate}
\item Let $Z\sim \mathcal{B}(\alpha,\beta)$, $\alpha,\beta>0$, with density $p(x)=\frac{1}{B(\alpha,\beta)}x^{\alpha-1}(1-x)^{\beta-1}$, $0<x<1$, where $B(\alpha,\beta)=\Gamma(\alpha)\Gamma(\beta)/\Gamma(\alpha+\beta)$ is the beta function.

\item The Stein equation is (see \cite{schoutens})
\begin{equation}\label{344}L_{\alpha,\beta}f(x)=x(1-x)f'(x)+(\alpha-(\alpha+\beta)x)f(x)=\tilde{h}(x).
\end{equation}
and, for $k\geq1$,
\begin{equation*}L_{\alpha+k,\beta+k}f^{(k)}(x)=h^{(k)}(x)+k(\alpha+\beta+k-1)f^{(k-1)}(x).
\end{equation*}
We have that $(L_k,T_k)=(L_{\alpha+k,\beta+k},(\alpha+\beta+2k)I)$ for $k\geq0$.  We can therefore bound the derivatives of the solution via Lemma \ref{lemma2.1}, with $a_k=\sum_{i=0}^k(\alpha+\beta+2i)=k(\alpha+\beta+k-1)$.

\item The solution to (\ref{344}) is
\begin{equation*}f(x)=\frac{1}{x^{\alpha}(1-x)^\beta}\int_0^x t^{\alpha-1}(1-t)^{\beta-1}\tilde{h}(t)\,\mathrm{d}t.
\end{equation*}

\item It was shown by \cite{dobler beta} that, if $h$ is bounded,
\begin{equation*}\|f\|\leq\frac{\|\tilde{h}\|}{2m(1-m)p(m)},
\end{equation*}
where $m$ is the median of the $\mathcal{B}(\alpha,\beta)$ distribution.  Also, for Lipschitz $h$ with Lipschitz constant $\|h'\|$,
\begin{equation*}\|f\|\leq\frac{\|h'\|}{\alpha+\beta} \quad \text{and} \quad \|f'\|\leq C(\alpha,\beta)\|h'\|,
\end{equation*}
where
\begin{align*}C(\alpha,\beta)&=2(\alpha+\beta)\begin{cases}B(\alpha,\beta), &\quad \alpha\leq1,\,\beta\leq1 \\
\alpha^{-1}, &\quad \alpha\leq1,\,\beta>1 \\
\beta^{-1}, &\quad \alpha>1,\,\beta\leq1 \\
\alpha^{-1}\beta^{-1}B(\alpha,\beta)^{-1}, &\quad \alpha>1,\,\beta>1
\end{cases} \quad \text{if $\alpha\not=\beta$ and} \\
C(\alpha,\alpha)&=\begin{cases}4, &\quad 0<\alpha<1 \\
\frac{2\alpha\sqrt{\pi}\Gamma(\alpha)}{\Gamma(\alpha+1/2)}, &\quad \alpha\geq1.
\end{cases}
\end{align*}
A bound of the form $\|f'\|\leq K(\alpha,\beta)\|h'\|$ was also obtained by \cite{goldstein4}, in which the constant $K(\alpha,\beta)$, for various parameter values, is sometimes better and sometimes worse than $C(\alpha,\beta)$.  An iterative approach with base case $\|f'\|\leq C(\alpha,\beta)\|h'\|$ was used by \cite{dobler beta} to obtain the bound, for $h\in C_b^{n}((0,1))$, $n\geq1$, 
\begin{align*}\|f^{(n)}\|&\leq \sum_{j=1}^nj!(\alpha+\beta+j-1)_{n-j}\bigg(\prod_{l=1}^nC(\alpha+l-1,\beta+l-1)\bigg)\|h^{(j)}\| .
\end{align*}

\item For the special case of the $\mathcal{B}(1/2,1/2)$ distribution, also known as the arcsine distribution, \cite{dobler1} showed that
\begin{equation*}\|f'\|\leq\min\bigg(\|h'\|\,,\, \frac{\|\tilde{h}\|}{\pi}\bigg),\quad\|f'\|\leq 4\|h'\|, \quad  \|f''\|\leq 6\pi\|h'\|+\frac{3}{2}\pi\|h''\|.
\end{equation*} 

\end{enumerate}

\subsubsection{Student's $t$ distribution}

\begin{enumerate}
\item For $\delta>0$ and $d>0$, the non-standardised Student's $t$-distribution has p.d.f.
\begin{equation} \label{fancy} p(x)=\frac{\Gamma(\frac{d+1}{2})}{\sqrt{\pi \delta^2}\Gamma(\frac{d}{2})}\bigg(1+\frac{x^2}{\delta^2}\bigg)^{-\frac{1}{2}(d+1)}, \quad x\in\mathbb{R}.
\end{equation}
In the case $d = \nu$, $\delta =\sqrt{\nu}$ the density (\ref{fancy}) is that of Student's $t$-distribution with $\nu$ degrees of freedom.

\item The Stein equation is (see \cite{gaunt thesis}, Section 7)
\begin{equation}\label{355}L_{d,\delta}f(x)=(\delta^2+x^2)f'(x)-(d-1)xf(x)=\tilde{h}(x)
\end{equation}
and, for $k\geq1$,
\begin{equation}\label{tsteinop}L_{d-2k,\delta}f^{(k)}(x)=h^{(k)}(x)+k(d-k)f^{(k-1)}(x).
\end{equation}
We have that $(L_k,T_k)=(L_{d-2k,\delta},(d-2k-1)I)$ for $k\geq0$.  We can therefore bound the derivatives of the solution via Lemma \ref{lemma2.1}, with $a_k=\sum_{i=0}^k(d-2i-1)=(k+1)(d-k-1)$.

\item The solution to (\ref{355}) is
\begin{equation*}  f(x)=(\delta^2+x^2)^{(d-1)/2}\int_{-\infty}^x(\delta^2+t^2)^{-(d+1)/2}\tilde{h}(t)\,\mathrm{d}t.
\end{equation*}

\item The density $p$ satisfies $(s(x)p(x))'=\tau(x)p(x)$ with $s(x)=\delta^2+x^2$ and $\tau(x)=-(d-1)x$, and therefore from Lemmas 1 and 3 of \cite{schoutens0} we have the bounds
\begin{equation}\label{t111}\|f\|\leq\frac{\sqrt{\pi}\Gamma(\frac{d}{2})}{2\delta\Gamma(\frac{d+1}{2})}\|\tilde{h}\|
\end{equation}
and
\begin{equation}\label{t222}\|f'\|\leq 2\|(\delta^2+x^2)^{-1}\|\|\tilde{h}\|=\frac{2}{\delta^2}\|\tilde{h}\|.
\end{equation}

\item We can use Lemma \ref{lemma2.1} with (\ref{t111}) and (\ref{t222}) as base cases to bound higher order derivatives of the solution.  

Suppose that $h\in C_b^{n}(\mathbb{R})$ and $d-2n>0$.  Then
\begin{equation*}\|f^{(n)}\|\leq \sum_{j=0}^n\bigg(\frac{\sqrt{\pi}}{2\delta}\bigg)^{n+1-j}\frac{n!}{j!}(d-n)_{n-j}A_j\|\tilde{h}^{(j)}\|,
\end{equation*}
where $A_j=\prod_{i=j}^n\frac{\Gamma(\frac{d}{2}-i)}{\Gamma(\frac{d+1}{2}-i)}$.

Suppose now that $h\in C_b^{n-1}(\mathbb{R})$ and $d-2(n-1)>0$.  Then, for $n=2k+1$,
\begin{equation*}\|f^{(2k+1)}\|\leq \sum_{j=0}^k\bigg(\frac{2}{\delta^2}\bigg)^{k-j+1}\frac{(2k)!!}{(2j)!!}(d-2k)_{k-j,2}\|\tilde{h}^{(2j)}\|
\end{equation*}
and, for $n=2k$,
\begin{align*}\|f^{(2k)}\|&\leq \sum_{j=1}^k\bigg(\frac{2}{\delta^2}\bigg)^{k-j+1}\frac{(2k-1)!!}{(2j-1)!!}(d-2k+1)_{k-j,2}\|h^{(2j-1)}\|\\
&\quad+\frac{\sqrt{\pi}\Gamma(\frac{d}{2})}{2\delta\Gamma(\frac{d+1}{2})}\bigg(\frac{2}{\delta^2}\bigg)^{k}(2k-1)!!(d-2k+1)_{k,2}\|\tilde{h}\|.
\end{align*}

\end{enumerate}

\subsubsection{Inverse-gamma distribution}

\begin{enumerate}
\item Consider an inverse-gamma random variable $Z$ with parameters $\alpha,\beta>0$ and density $p(x)=\frac{\beta^\alpha}{\Gamma(\alpha)}x^{-\alpha-1}e^{-\beta/x}$, $x>0$.

\item The Stein equation is (see \cite{koudou})
\begin{equation}\label{3ii}L_{\alpha,\beta}f(x)=x^2f'(x)+(\beta-(\alpha-1)x)f(x)=\tilde{h}(x)
\end{equation}
and, for $k\geq1$,
\begin{equation*}L_{\alpha-2k,\beta}f^{(k)}(x)=h^{(k)}(x)+k(\alpha-k)f^{(k-1)}(x).
\end{equation*}
We have that $(L_k,T_k)=(L_{\alpha-2k,\beta},(\alpha-2k-1)I)$ for $k\geq0$.  We can therefore bound the derivatives of the solution via Lemma \ref{lemma2.1}, with $a_k=\sum_{i=0}^k(\alpha-2i-1)=(k+1)(\alpha-k-1)$.

\item The solution to (\ref{3ii}) is
\begin{equation*}  f(x)=x^{\alpha-1}e^{\beta/x}\int_{0}^x\tilde{h}(t)t^{-\alpha-1}e^{-\beta/t}\,\mathrm{d}t.
\end{equation*}

\item The density $p$ satisfies $(s(x)p(x))'=\tau(x)p(x)$ with $s(x)=x^2$ and $\tau(x)=\beta-(\alpha-1)x$. Therefore using Lemma 1 of \cite{schoutens0} together with the fact that $\mathbb{E}Z=\frac{\beta}{\alpha-1}$ for $\alpha>1$, we have the bound, for $\alpha>1$,
\begin{equation}\label{t111i}\|f\|\leq \frac{\Gamma(\alpha)}{\beta}\bigg(\frac{e}{\alpha-1}\bigg)^{\alpha-1}\|\tilde{h}\|.
\end{equation}

\item We can use Lemma \ref{lemma2.1} with base case (\ref{t111i}) to bound higher order derivatives of the solution.  

Suppose that $h\in C_b^{n}(\mathbb{R_+})$ and $\alpha>2n+1$.  Then
\begin{equation*}\|f^{(n)}\|\leq \sum_{j=0}^n\frac{n!}{j!}(\alpha-n)_{n-j}B_j\|\tilde{h}^{(j)}\|,
\end{equation*}
where $B_j=\prod_{i=j}^n\frac{\Gamma(\alpha-2i)}{\beta}\left(\frac{e}{\alpha-2i-1}\right)^{\alpha-2i-1}$.

\end{enumerate}

\subsubsection{PRR (P\"ekoz-R\"ollin-Ross) distribution}

\begin{enumerate}
\item Considered the PRR distribution (introduced by \cite{pekoz}) with density
\begin{equation*}\label{kummeru}\kappa_s(x)=\Gamma(s)\sqrt{\frac{2}{s\pi}}\exp\bigg(-\frac{x^2}{2s}\bigg)U\bigg(s-1,\frac{1}{2},\frac{x^2}{2s}\bigg), \quad x>0,\: s\geq \frac{1}{2},
\end{equation*}
where $U(a,b,x)$ denotes the confluent hypergeometric function of the second kind (\cite{olver}, Chapter 13).

\item The Stein equation is (see \cite{pekoz}) 
\begin{equation}\label{366}L_sf(x)=sf''(x)-xf'(x)-2(s-1)f(x)=\tilde{h}(x)
\end{equation}
and, for $k\geq1$,
\begin{equation*}L_sf^{(k)}(x)=h^{(k)}(x)+kf^{(k)}(x).
\end{equation*}
We have that $(L_k,T_k)=(L_{s},D)$ for $k\geq0$.  We can therefore bound the derivatives of the solution via Lemma \ref{lemma2.2}, with $a_k=k+1$.

\item The solution to (\ref{366}) is
\begin{equation*}f(x)=\frac{1}{s}V_s(x)\int_0^x\frac{1}{V_s(y)\kappa_s(y)}\int_0^y\tilde{h}(t)\kappa_s(t)\,\mathrm{d}t\,\mathrm{d}y,
\end{equation*}
where $V_s(x)=U\left(s-1,\frac{1}{2},\frac{x^2}{2s}\right)$.

\item The following bounds were obtained by \cite{pekoz}.  For $s\geq1$,
\begin{equation*}\|f'\|\leq\sqrt{2\pi}\|h\|, \quad \|f''\|\leq2\bigg(\pi\sqrt{s}+\frac{1}{s}\bigg)\|h\|, \quad \|f^{(3)}\|\leq 8\bigg(s+\frac{1}{4}+\sqrt{\frac{\pi}{2}}\bigg)\|h'\|
\end{equation*}
and, for $s=1/2$,
\begin{equation*}\|f''\|\leq4\|h\|, \quad \|f^{(3)}\|\leq 4\|h'\|.
\end{equation*}

\item We can apply Lemma \ref{lemma2.2} with $\|f'\|\leq\sqrt{2\pi}\|h\|$ and $\|f''\|\leq4\|h\|$ as base cases to bound higher order derivatives of the solution.  

Suppose $h\in C_b^{n-1}(\mathbb{R}_+)$.  Then, for $s\geq 1$,
\begin{equation*}\|f^{(n)}\|\leq (n-1)!\sum_{j=0}^{n-1}\frac{(2\pi)^{(n-j)/2}}{j!}\|h^{(j)}\|.
\end{equation*}
Suppose now that $h\in C_b^{n-2}(\mathbb{R}_+)$.  Then, for $n=2k+1$,
\begin{equation*}\|f^{(2k+1)}\|\leq \sqrt{2\pi}C^{k}(2k-1)!!\|h\|+ \sum_{j=1}^{k-1}C^{k-j+1}\frac{(2k-1)!!}{(2j-1)!!}\|h^{(2j-1)}\|
\end{equation*}
and, for $n=2k$,
\begin{equation*}\|f^{(2k)}\|\leq  \sum_{j=0}^{k-1}C^{k-j+1}\frac{(2k)!!}{(2j)!!}\|h^{(2j)}\|,
\end{equation*}
where $C=4$ if $s=1/2$ and $C=2(\sqrt{\pi}s+s^{-1})$ if $s\geq 1$.

\end{enumerate}

\subsubsection{Variance-gamma distribution}

\begin{enumerate}
\item The $\mathrm{VG}(r,\theta,\sigma,\mu)$ distribution with parameters $r > 0$, $\theta \in \mathbb{R}$, $\sigma >0$, $\mu \in \mathbb{R}$ has p.d.f.
\begin{equation}\label{vgdefn}p(x) = \frac{1}{\sigma\sqrt{\pi} \Gamma(\frac{r}{2})} e^{\frac{\theta}{\sigma^2} (x-\mu)} \bigg(\frac{|x-\mu|}{2\sqrt{\theta^2 +  \sigma^2}}\bigg)^{\frac{r-1}{2}} K_{\frac{r-1}{2}}\bigg(\frac{\sqrt{\theta^2 + \sigma^2}}{\sigma^2} |x-\mu| \bigg), 
\end{equation}
where $x \in \mathbb{R}$ and the modified Bessel function of the second kind is given, for $x>0$, by $K_\nu(x)=\int_0^\infty e^{-x\cosh(t)}\cosh(\nu t)\,\mathrm{d}t$ (see \cite{olver}).  The support of the variance-gamma distributions is $\mathbb{R}$ when $\sigma>0$, but in the limit $\sigma\rightarrow 0$ the support is the region $(\mu,\infty)$ if $\theta>0$, and is $(-\infty,\mu)$ if $\theta<0$.  The class of variance-gamma distributions contains many classical distributions as special cases (see \cite{gaunt vg} for a list of these cases); in particular,
\begin{align*}N(\mu,\sigma^2)&\stackrel{\mathcal{D}}{=}\lim_{r\rightarrow\infty}\mathrm{VG}(r,0,\sigma/\sqrt{r},\mu), \\
\Gamma(r,\lambda)&\stackrel{\mathcal{D}}{=}\lim_{\sigma\downarrow0}\mathrm{VG}(2r,(2\lambda)^{-1},\sigma,0), \\
\mathrm{Laplace}(\mu,\sigma)&\stackrel{\mathcal{D}}{=}\mathrm{VG}(2,0,\sigma,\mu), \\
\mathrm{PN}(2,\sigma^2)&\stackrel{\mathcal{D}}{=}\mathrm{VG}(1,0,\sigma,0),
\end{align*}
where $\mathrm{Laplace}(\mu,\sigma)$ denotes a Laplace distribution with density $\frac{1}{2\sigma}\exp\big(-\frac{|x-\mu|}{\sigma}\big)$, $x\in\mathbb{R}$, and $\mathrm{PN}(2,\sigma_X^2\sigma_Y^2)$ is the distribution of the product of the independent random variables $X\sim N(0,\sigma_X^2)$ and $Y\sim N(0,\sigma_Y^2)$.

\item For simplicity, consider the case $\mu=0$.  The Stein equation is (see \cite{gaunt vg})
\begin{equation}\label{377}L_{r,\theta,\sigma}f(x)=\sigma^2xf''(x)+(\sigma^2r+2\theta x)f'(x)-(r\theta-x)f(x)=\tilde{h}(x).
\end{equation}
Note that the $\mathrm{VG}(r,\theta,\sigma,0)$ Stein equation reduces to the standard normal (\ref{311}) and gamma (\ref{322}) Stein equations in the appropriate limits.

For $k\geq1$, we have
\begin{equation*}L_{r+k,\theta,\sigma}f^{(k)}(x)=h^{(k)}(x)+kf^{(k-1)}(x)+k\theta f^{(k)}(x).
\end{equation*}
We have that $(L_k,T_k)=(L_{r+k,\theta,\sigma}f^{(k)},I+\theta D)$ for $k\geq0$.  When $\theta=0$ we can bound the derivatives of solution using Lemma \ref{lemma2.1}, with $a_k=k+1$.  Otherwise, we can obtain estimates via Lemma \ref{lemma2.3}, with $a_k=k+1$ and $b_k=\theta(k+1)$.

\item The solution to (\ref{377}) is
\begin{align}  f(x) &=-\frac{e^{-\theta x/\sigma^2}}{\sigma^2|x|^{\nu}}K_{\nu}\bigg(\frac{\sqrt{\theta^2+\sigma^2}}{\sigma^2}|x|\bigg) \int_0^x e^{\theta t/\sigma^2} |t|^{\nu} I_{\nu}\bigg(\frac{\sqrt{\theta^2+\sigma^2}}{\sigma^2}|t|\bigg) \tilde{h}(t) \,\mathrm{d}t \nonumber \\
\label{vgsolngeneral} &\quad-\frac{e^{-\theta x/\sigma^2}}{\sigma^2|x|^{\nu}}I_{\nu}\bigg(\frac{\sqrt{\theta^2+\sigma^2}}{\sigma^2}|x|\bigg) \int_x^{\infty} e^{\theta t/\sigma^2} |t|^{\nu} K_{\nu}\bigg(\frac{\sqrt{\theta^2+\sigma^2}}{\sigma^2}|t|\bigg)\tilde{h}(t)\,\mathrm{d}t,
\end{align}
where $\nu=\frac{r-1}{2}$ and $I_{\nu}(x)=\sum_{k=0}^{\infty}\frac{1}{k!\Gamma(\nu+k+1)}\left(\frac{x}{2}\right)^{\nu+2k}$ is a modified Bessel function of the first kind (see \cite{olver}).

\item For the case $\theta=0$, \cite{gaunt vg} obtained the estimates
\begin{eqnarray}\label{vgf0}\|f\|&\leq& \frac{1}{\sigma}\bigg(\frac{1}{r}+\frac{\pi\Gamma(\frac{r}{2})}{2\Gamma(\frac{r+1}{2})}\bigg)\|\tilde{h}\|,  \\
\label{vgf1}\|f'\| &\leq& \frac{2}{\sigma^2r}\|\tilde{h}\|.
\end{eqnarray}
as well as uniform bounds for $f''$, $f^{(3)}$ and $f^{(4)}$.

\item An application of Lemma \ref{lemma2.1} with (\ref{vgf0}) and (\ref{vgf1}) as base cases allows us to obtain bounds for higher order derivatives.  

Suppose $h\in C_b^{n-1}(\mathbb{R})$.  Then, for $n=2k+1$,
\begin{equation*}\|f^{(2k+1)}\|\leq \sum_{j=0}^k\frac{(2k)!!}{(2j)!!}\frac{1}{(r+2j+2)_{k-j+1,2}}\bigg(\frac{2}{\sigma^2}\bigg)^{k-j+1}\|\tilde{h}^{(2j)}\|
\end{equation*}
and, for $n=2k$,
\begin{align*}\|f^{(2k)}\|&\leq \sum_{j=1}^k\frac{(2k-1)!!}{(2j-1)!!}\frac{1}{(r+2j+1)_{k-j+1,2}}\bigg(\frac{2}{\sigma^2}\bigg)^{k-j+1}\|h^{(2j-1)}\|\\
&\quad+\frac{(2k-1)!!2^k}{(r+1)_{k,2}\sigma^{2k+1}}\bigg(\frac{1}{r}+\frac{\pi\Gamma(\frac{r}{2})}{2\Gamma(\frac{r+1}{2})}\bigg)\|\tilde{h}\|.
\end{align*}

\item Now, consider the case of general $\theta$.  In Section 3.2.4, we obtain the estimates
\begin{eqnarray}\label{vgsolnunibound}\|f\|&\leq&\frac{\|\tilde{h}\|}{\sqrt{\theta^2+\sigma^2}}\bigg(\frac{2}{r}+C(r,\theta,\sigma)\bigg), \\
\label{vgsolnunibound1}\|f'\|&\leq&\frac{\|\tilde{h}\|}{\sigma^2}\bigg(\frac{2}{r}+C(r,\theta,\sigma)\bigg),
\end{eqnarray}
where
\begin{equation*}C(r,\theta,\sigma)=\begin{cases} \displaystyle \frac{\sqrt{\pi}\Gamma\left(\frac{r}{2}\right)}{\Gamma\left(\frac{r+1}{2}\right)}\bigg(1+\frac{\theta^2}{\sigma^2}\bigg)^{r/2}, & \:  r\geq2, \\
\displaystyle 6\Gamma\left(\frac{r}{2}\right)\bigg(1-\frac{1}{(1+\sigma^2/\theta^2)^{1/2}}\bigg)^{-r/2}, & \:   0<r<2. \end{cases}
\end{equation*}

\item We can use Lemma \ref{lemma2.3} with (\ref{vgsolnunibound}) and (\ref{vgsolnunibound1}) as base cases to obtain bounds for higher order derivatives.

Let $B_j=\frac{2}{r+j}+C(r+j,\theta,\sigma)$.  Then, for $h\in C_b^{n-1}(\mathbb{R})$,
\begin{align*}
 \|f^{(n)}\|&\leq\frac{1}{\sigma^2}\sum_{j=0}^{n-1} \bigg(B_{n-j}\sum_{l\in I_j} A_{j,l}^{(n)}\bigg)\|h^{(n-j)}\|\\
 &\quad+\frac{B_0}{\sigma^2}\bigg(\sum_{l\in I_{n}}A_{n,l}^{(n)}+\frac{B_1}{\sqrt{\theta^2+\sigma^2}}\sum_{l\in I_{n-1}}A_{n-1,l}^{(n)}\bigg)\|\tilde{h}\|.
\end{align*}
where 
\begin{equation*}
 A_{j,l}^{(n)}:=\sum_{M\in\mathcal{S}_{j,l}^{(n)}}\prod_{i\in M\setminus\{n-j\}}\big(|\theta| 1_{\{i-1\in M\}}+ 1_{\{i-1\notin M\}}\big)\frac{iB_i}{\sigma^2},
\end{equation*}
and the sets $I_j$ and $\mathcal{S}_{j,l}^{(n)}$ are defined as in Lemma \ref{lemma2.3}.

In particular, applying the iterative technique with (\ref{vgsolnunibound}) and (\ref{vgsolnunibound1}) as base cases yields the following simple bounds:
\begin{eqnarray*}\|f''\|&\leq&\frac{B_{1}}{\sigma^2}\|h'\|+\frac{B_0B_{1}}{\sigma^2}\bigg[\frac{|\theta|}{\sigma^2}+\frac{1}{\sqrt{\theta^2+\sigma^2}}\bigg]\|\tilde{h}\|, \\
\|f^{(3)}\|&\leq&\frac{B_{2}}{\sigma^2}\|h''\|+\frac{2|\theta| B_{1}B_{2}}{\sigma^4}\|h'\|+\frac{2B_{0}B_{2}}{\sigma^4}\bigg[1+\frac{|\theta| B_{1}}{\sqrt{\theta^2+\sigma^2}}+\frac{\theta^2 B_{1}}{\sigma^2}\bigg]\|\tilde{h}\|, \\
\|f^{(4)}\|&\leq&\frac{B_{3}}{\sigma^2}\|h^{(3)}\|+\frac{3|\theta| B_{2}B_{r+3}}{\sigma^4}\|h''\|+\frac{3B_{1}B_{3}}{\sigma^4}\bigg[1+\frac{2\theta^2B_{2}}{\sigma^2}\bigg]\|h'\|\\
&&+\frac{3B_{3}}{\sigma^4}\bigg[\frac{|\theta| B_0}{\sigma^2}\bigg(B_{1}+2\theta B_{2}+\frac{2\theta^2B_{1}B_{2}}{\sigma^2}\bigg) +\frac{B_{1}}{\sqrt{\theta^2+\sigma^2}}\bigg(1+\frac{2\theta^2B_{2}}{\sigma^2}\bigg)\bigg]\|\tilde{h}\|.
\end{eqnarray*}

\end{enumerate}

\subsubsection{Density proportional to $e^{-x^4/12}$}

\begin{enumerate}
\item Consider a probability distribution with density proportional to $e^{-x^4/12}$.  We quote from \cite{cs11} that the exact form of the density $p$ of this distribution is 
\begin{equation*}
 p(x)=\frac{\sqrt{2}}{3^{1/4}\Gamma(1/4)}e^{-x^4/12}, \quad x\in\mathbb{R}.
\end{equation*} 

The limiting distribution of the properly scaled total magnetization in the Curie-Weiss model at the critical inverse temperature 
$\beta=1$ was shown to have this density in \cite{en78}. In \cite{el10} and \cite{cs11} a version of Stein's method of exchangeable pairs was used to prove that also a Berry-Ess\'{e}en rate of order 
$n^{-1/2}$ holds for this limit theorem. As in neither of the works \cite{el10} and \cite{cs11} a constant for this rate is aimed at, no numerical bounds on the solution to the Stein equation 
have been stated, yet. However, in the appendix section of \cite{el10}, actually some numerical bounds are proved. 

We will firstly use these and the general bounds from \cite{dobler beta} to provide concrete bounds on the 
solution to the Stein equation and, then, use these as starting points for our iterative method 
to derive bounds for derivatives of all orders. 

\item The Stein equation is (see \cite{el10} and \cite{cs11})
\begin{equation}\label{411}
Lf(x)= f'(x)-\frac{x^3}{3}f(x)=\tilde{h}(x).
\end{equation}
We have
\begin{eqnarray}
Lf'(x)&=&h'(x)+x^2f(x)=:h_1(x),\label{cwi1} \\
Lf''(x)&=&h''(x)+2xf(x)+2x^2f'(x)=:h_2(x)\label{cwi2},
\end{eqnarray}
and, for $k\geq3$,
\begin{align*}Lf^{(k)}(x)&=h^{(k)}(x)+\frac{1}{3}k(k-1)(k-2)f^{(k-3)}(x)+k(k-1)xf^{(k-2)}(x)+kx^2f^{(k-1)}(x)\\
&=:h_k(x).
\end{align*}
We are unable to apply either of Lemmas \ref{lemma2.1}--\ref{lemma2.3} to bound the derivatives of the solution.  However, we can still apply the iterative technique to bound the derivatives; the details are given in Section 3.2.5.

\item The solution to (\ref{411}) is
\begin{equation*}  f(x)=e^{x^4/12}\int_{-\infty}^x \tilde{h}(t)e^{-t^4/12}\,\mathrm{d}t.
\end{equation*}

\item If $h$ is bounded and measurable, then it holds that 
\begin{equation}\label{elbounds}
 |f(x)|\leq\min\left(\frac{3^{1/4}\Gamma(1/4)}{2\sqrt{2}},\,\frac{3}{|x|^3}\right)\|\tilde{h}\|,\quad \|f'\|\leq2\|\tilde{h}\|.
\end{equation}
These bounds are not stated explicitly, but are actually proved in the appendix section of \cite{el10}.
If $h$ is Lipschitz-continuous, then we can prove that 
\begin{equation}\label{dbounds1}
 \|f\|\leq \frac{\sqrt{3\pi}}{2}\|h'\|\,,\quad\|f'\|\leq 2^{1/2}3^{1/4}\Gamma(1/4)\|h'\|,\quad\|f''\|\leq 4\|h'\|.
\end{equation}
More precisely, we can even show that for each $x\in\mathbb{R}$ we have 
\begin{equation}\label{dbounds2}
 |f(x)|\leq\min\bigg(\frac{\sqrt{3\pi}}{2},\,\frac{3}{x^2}\bigg)\|h'\|,
 \quad |f'(x)|\leq \min\bigg(2^{1/2}3^{1/4}\Gamma(1/4),\,\frac{12}{|x|^3}\bigg)\|h'\|.
\end{equation}
We would like to stress that the numerical bounds for Lipschitz test functions $h$ in \eqref{dbounds1} are original and the first ones obtained for this distribution. 
The bound on $f''$ in \eqref{dbounds1} will be proved by standard, but pretty involved calculations. It should thus not be left unmentioned that, using the fairly easy second bound in \eqref{elbounds} 
and our iterative procedure, we can easily obtain the bound 
\begin{equation}\label{alcwb}
 \|f''\|\leq 8\|h'\|,
\end{equation}
which has a slightly worse constant. This is completely analogous to the case of the normal distribution (see Section 3.2.1).

For bounds on derivatives of general order let $h\in C_b^n(\mathbb{R})$.
We obtain the following generalization of the first bound in \eqref{elbounds}:
\begin{equation}\label{genelbound1}
 |f^{(n)}(x)|\leq\min\left(\frac{3^{1/4}\Gamma(1/4)}{2\sqrt{2}},\,\frac{3}{|x|^3}\right)\sum_{k=0}^n a_{n-k}^{(n)}\|\tilde{h}^{(n-k)}\|\,,\quad x\in\mathbb{R}.
\end{equation}
Here, the nonnegative real numbers $a_j^{(n)}$, $n\in\mathbb{N}_0$, $0\leq j\leq n$ are given as follows: For each $n$ we have that 
\begin{equation}\label{inian}
 a_n^{(n)}=1\,,\quad a_{n-1}^{(n)}=\frac{3n}{\bigl(6c_1\bigr)^{1/3}}\,,\quad a_{n-2}^{(n)}=\frac{12n(n-1)}{\bigl(6c_1\bigr)^{2/3}}
 \end{equation}
whenever these quantities are defined, and, for $3\leq k\leq n$,
\begin{equation}\label{recan}
 a_{n-k}^{(n)}=\frac{3n}{\bigl(6c_1\bigr)^{1/3}}a_{n-k}^{(n-1)}+\frac{3n(n-1)}{\bigl(6c_1\bigr)^{2/3}}a_{n-k}^{(n-2)}+\frac{n(n-1)(n-2)}{6c_1}a_{n-k}^{(n-3)},
\end{equation}
where we write
\begin{equation*}
 c_1=\frac{\sqrt{2}}{3^{1/4}\Gamma(1/4)}
\end{equation*}
for the normalizing constant. Since we can also show that 
\begin{equation}\label{hnbound}
 \|h_n\|\leq \sum_{k=0}^n a_{n-k}^{(n)}\|\tilde{h}^{(n-k)}\|,
 \end{equation}
from the second bound in \eqref{elbounds} we obtain the following generalised version of this bound:  
 \begin{equation}\label{genelbound2}
 \|f^{(n+1)}\|\leq 2\sum_{k=0}^n a_{n-k}^{(n)}\|\tilde{h}^{(n-k)}\|.
\end{equation}

\end{enumerate}

\subsubsection{Multivariate normal distribution}

\begin{enumerate}
\item Let $\mathrm{Z}\sim \mathrm{MVN}(0,\Sigma)$ with density $p(x)=\frac{1}{(2\pi)^{d/2}\sqrt{\mathrm{det}(\Sigma)}}\exp\left(-\frac{1}{2}x^T\Sigma^{-1}x\right)$, $x\in\mathbb{R}^d$.  The covariance matrix $\Sigma\in \mathbb{R}^{d\times d}$ is nonnegative-definite.

\item The Stein equation is (see \cite{goldstein1})
\begin{equation}\label{422}Lf(x)=\sum_{i=1}^d\sum_{j\not=i}^d\sigma_{ij}\frac{\partial^2 f}{\partial x_i\partial x_j}(x)-\sum_{i=1}^dx_i\frac{\partial f}{\partial x_i}(x)=\tilde{h}(x),
\end{equation}
where $\sigma_{ij}=(\Sigma)_{ij}$, and, for $k\geq1$,
\begin{equation*}L\bigg(\frac{\partial^k f(x)}{\prod_{j=1}^k\partial x_{i_j}}\bigg)=\frac{\partial^kh(x)}{\prod_{j=1}^k\partial x_{i_j}}+k\frac{\partial^kf(x)}{\prod_{j=1}^k\partial x_{i_j}}.
\end{equation*}
We are unable to apply either of Lemmas \ref{lemma2.1}--\ref{lemma2.3} to bound the derivatives of the solution.  However, we can still apply the iterative technique to bound the derivatives; the details are given in Section 3.2.6.

\item The solution to (\ref{422}) is
\begin{equation*} f(x)=-\int_0^\infty [\mathbb{E}h(e^{-s}x+\sqrt{1-e^{-2s}}\mathrm{Z})-\mathbb{E}h(\mathrm{Z})]\,\mathrm{d}s.
\end{equation*}

\item It was shown by \cite{goldstein1} that, if $\Sigma$ is nonnegative-definite and $h\in C_b^{n}(\mathbb{R}^d)$, then the solution satisfies
 \begin{equation}\label{cheque} \bigg\|\frac{\partial^n f}{\prod_{j=1}^n\partial x_{i_j}}\bigg\|\leq \frac{1}{n}\bigg\|\frac{\partial^n h}{\prod_{j=1}^n\partial x_{i_j}}\bigg\|, \quad n\geq 1.
\end{equation}

Suppose now that $\Sigma$ is positive-definite matrix and that $h$ is bounded.  Then, the following estimate was obtained by \cite{gaunt rate}: 
\begin{equation}\label{mvnbound}\bigg\|\frac{\partial f}{\partial x_i}\bigg\|\leq \sqrt{\frac{\pi}{2}}\Bigg[\sum_{j=1}^d\sigma_{ij}^2\Bigg]^{1/2}\|\tilde{h}\|,
\end{equation}
If $h\in C_b^{n-1}(\mathbb{R}^d)$, $n\geq 2$, then
\begin{equation}\label{charm}\bigg\|\frac{\partial^n f}{\prod_{j=1}^n\partial x_{i_j}}\bigg\|\leq\frac{\Gamma(\frac{n}{2})}{\sqrt{2}\Gamma(\frac{n+1}{2})}\min_{1\leq l\leq n}\Bigg\{\Bigg[\sum_{j=1}^d\sigma_{i_lj}^2\Bigg]^{1/2} \bigg\|\frac{\partial^{n-1} h}{\prod_{\stackrel{1\leq j\leq n}{j\not=l}}\partial x_{i_j}}\bigg\|\Bigg\}.
\end{equation}

Analogous bounds to (\ref{cheque}) and (\ref{charm}) were obtained for the $n$-th derivatives of the solution $f$, viewed as $n$-linear forms, by \cite{meckes} and \cite{gaunt rate} respectively.

\end{enumerate}

\subsection{Further comments and proofs}

In this section, we provide additional comments to complement the results of Section 3.1, and we also provide insights into the application of the iterative technique in practice.

\subsubsection{Normal distribution}

Using elementary calculations, \cite{stein2} obtained the bounds (\ref{snbound}) for the solution (\ref{311soln}) of the standard normal Stein equation. However, this approach quickly becomes tedious for bounding higher order derivatives.  The first progress with regard to bounding higher order derivatives came from the observation by \cite{barbour2} that the standard normal Stein operator is the generator of an Ornstein-Uhlenbeck process (if we take $f=g'$).  Applying standard theory of Markov processes leads to a solution of the form (\ref{gensoln}):
\begin{equation}\label{gsolnnor}g(x)=-\int_0^\infty [\mathbb{E}h(e^{-s}x+\sqrt{1-e^{-2s}}Z)-\mathbb{E}h(Z)]\,\mathrm{d}s.
\end{equation}
If $h\in C_b^k(\mathbb{R})$ then, by the dominated convergence theorem,
\begin{equation}\label{norderivg}f^{(k-1)}(x)=g^{(k)}(x)=-\int_0^\infty e^{-ks}\mathbb{E}[h^{(k)}(e^{-s}x+\sqrt{1-e^{-2s}}Z)]\,\mathrm{d}s,
\end{equation}
and the bound (\ref{snbound1}) now follows easily.  The bound (\ref{snbound2}) results from an integration by parts on (\ref{norderivg}).  However, one has to work significantly harder to obtain (\ref{snbound3}) and quite involved calculations were required to achieve the bound.

Let us now consider the application of our iterative method to the standard normal Stein equation.  Using part (i) of Lemma \ref{lemma2.1} with $\|f\|\leq\sqrt{\frac{\pi}{2}}\|\tilde{h}\|$ as the base case, we have that, for $C_b^{(n)}(\mathbb{R})$,
\begin{equation*}\|f^{(n)}\|\leq\sum_{j=0}^n\bigg(\frac{\pi}{2}\bigg)^{(n-j+1)/2}\frac{n!}{j!}\|\tilde{h}^{(j)}\|,
\end{equation*}
and an application of part (ii) of Lemma \ref{lemma2.1} with base case $\|f'\|\leq 2\|\tilde{h}\|$ leads to a bound on $f^{(n)}$ involving one fewer derivative of $h$.  These bounds are much worse than the best bounds from the literature, as given in Section 3.1.1.  However, the iterative technique does allow us to deduce bounds for derivatives of any order of the solution given just the easily derived bounds $\|f\|\leq\sqrt{\frac{\pi}{2}}\|\tilde{h}\|$ and $\|f'\|\leq 2\|\tilde{h}\|$.

An interesting application of the iterative method is that it allows us to easily deduce a bound of the form $\|f^{(n+1)}\|\leq C\|h^{(n)}\|$, where $C$ is a universal constant, given just the bounds $\|f'\|\leq2\|\tilde{h}\|$ and $\|f^{(n-1)}\|\leq\frac{1}{n}\|h^{(n)}\|$.  Recall that, for $k\geq1$,
\begin{equation*}Lf^{(k)}(x)=h^{(k)}(x)+kf^{(k-1)}(x).
\end{equation*}
Then, an application of the iterative procedure with these bounds yields
\begin{equation*}\|f^{(n+1)}\|\leq 2\|h^{(n)}(x)+nf^{(n-1)}(x)\|\leq 2(\|h^{(n)}\|+n\|f^{(n-1)}\|)\leq 4\|h^{(n)}\|.
\end{equation*}
The constant $4$ is worse than the constant $2$ in the bound (\ref{snbound3}) obtained by \cite{daly}, but our derivation is considerably simpler.  This example also illustrates how the iterative technique can be used to deduce bounds of the form $\|f^{(n+1)}\|\leq C\|h^{(n)}\|$ from bounds of the form $\|f^{(n)}\|\leq K\|h^{(n)}\|$.  This is of course useful in practice, because bounds involving fewer derivatives of the test function $h$ are often desirable.

\subsubsection{Gamma distribution}

As is the case with the normal distribution, bounding derivatives of the solution of the gamma Stein equation using the representation (\ref{322soln}) of the solution is tedious.  The bound (\ref{schgam}) for the solution can be obtained by this approach, but bounds for derivatives of the solution have not been established through this representation of the solution.  However, \cite{luk} used a generator approach, similar to that described for the normal distribution, to obtain a formula for the $n$-th derivative of the solution of the Stein equation, from which the bound (\ref{chisquarebound1}) follows.  Through a technical calculation, \cite{gaunt thesis} obtained (\ref{chisquarebound2}) (this improved an earlier bound of \cite{pickett}).  The bound (\ref{chisquarebound3}) was obtained by an iterative approach similar to that described in this paper.

As was the case for the normal distribution, a direct application of part (i) of Lemma \ref{lemma2.1} with (\ref{schgam}) as the base case would yield bounds on the higher order derivatives of the solution that are much worse than the best bounds from the literature.  However, the difficulty of bounding higher order derivatives is again greatly reduced, because one only requires the bound (\ref{schgam}) to achieve such bounds.

For the gamma distribution, we can again use the iterative technique to obtain bounds of the form $\|f^{(n)}\|\leq C\|h^{(n)}\|$ from bounds of the form $\|f^{(n)}\|\leq K\|h^{(n+1)}\|$.  Using an iterative approach with inequalities (\ref{schgam}) and (\ref{chisquarebound1}) gives
\begin{align}\|f^{(n)}\|&\leq\frac{e^{r+n}\Gamma(r+n)}{(r+n)^{r+n}}\|h^{(n)}(x)+n\lambda f^{(n-1)}(x)\|\nonumber\\
&\leq\frac{e^{r+n}\Gamma(r+n)}{(r+n)^{r+n}}(\|h^{(n)}\|+n\lambda \|f^{(n-1)}\|)\nonumber\\
\label{gammarr}&\leq \frac{2e^{r+n}\Gamma(r+n)}{(r+n)^{r+n}}\|h^{(n)}\|.
\end{align}

Thus, the iterative method has yielded a bound of similar form to (\ref{chisquarebound2}), but with a much simpler proof.  Using Stirling's approximation, we see that (\ref{gammarr}) is of order $(r+n)^{-1/2}$ as $r+n\rightarrow\infty$, which is the same order as (\ref{chisquarebound2}).  However, the constant is better in (\ref{chisquarebound2}) than (\ref{gammarr}) for all $n\geq1$, whilst the constant in (\ref{gammarr}) is better than bound (\ref{chisquarebound2}) in the case $n=0$. 

\subsubsection{Beta, Student's $t$, inverse-gamma and the PRR distribution}

Unlike the normal and gamma distributions, there are currently no simple formulas in the literature for general $n$-th order derivatives of the solutions of the beta, Student's $t$, inverse-gamma and PRR Stein equations.  However, for each of these distributions uniform bounds are available for the solution and some of its lower order derivatives (beta and PRR) or it is straightforward to derive such bounds (Student's $t$ and inverse-gamma).  Therefore, the iterative method of this paper is particularly useful for these distributions.  

From a straightforward application of Lemmas \ref{lemma2.1} and \ref{lemma2.2} we arrive at our bounds for $n$-th order derivatives of the solutions of the non-standardised Student's $t$, inverse-gamma and PRR Stein equations respectively.  We could also apply Lemma \ref{lemma2.1} to bound the higher order derivatives of the solution of the beta Stein equation; however, this was already done by \cite{dobler beta}.  The bound of \cite{dobler beta} is slightly different from the one that would arise from an application of Lemma \ref{lemma2.1}, as in that work the base case was $\|f'\|\leq C(\alpha,\beta)\|h'\|$, rather than a base case of the form $\|f\|\leq K(\alpha,\beta)\|\tilde{h}\|$.

The non-standardised Student's $t$ example illustrates a pertinent point regarding the application of the iterative method.  Suppose that we are interested in bounding the derivatives of solution of the classical Student's $t$ Stein equation ($d=\nu$, $\delta=\sqrt{\nu}$).  By (\ref{tsteinop}), we have that, for $k\geq1$, the solution satisfies
\begin{equation*}L_{\nu-2k,\sqrt{\nu}}f^{(k)}(x)=h^{(k)}(x)+k(\nu-k)f^{(k-1)}(x).
\end{equation*} 
It is therefore clear that in order to use the iterative approach to bound higher order derivatives of the solution of the Student's $t$ Stein equation, we must work with the Stein equation for the more general non-standardised Student's $t$ distribution.  It was also observed by \cite{dobler beta} that in order to apply the iterative method to the exponential Stein equation, one must must work with the more general gamma Stein equation.

\subsubsection{Variance-gamma distribution}

For the case $\theta=0$, we were able to bound $n$-th order derivatives of the solution of the variance-gamma Stein equation via Lemma \ref{lemma2.1} with inequalities (\ref{vgf0}) and (\ref{vgf1}) as base cases.  Previously, only uniform bounds existed for the solution and its first four derivatives (see \cite{gaunt vg}).  Our bounds for the $n$-th order derivatives decay at a rate of order $(r+n)^{-1}$ as $r+n\rightarrow\infty$; better than the order $r^{-1/2}$ rate obtained by \cite{gaunt vg}, and our bounds also have smaller constants.  Consequently, our bounds allows for an improvement of the $O(r^{1/2}(m^{-1}+n^{-1}))$ bound of Theorem 4.9 of \cite{gaunt vg} to a $O(m^{-1}+n^{-1})$ bound.

For the case of general $\theta$, only modest progress had been made on the problem of bounding the derivatives of the solution of the Stein equation.  It was shown by \cite{gaunt thesis} and \cite{gaunt vg} that the solution and its first derivative are bounded if the test function $h$ is bounded, although explicit constants were not given.  The complicated form of the solution (\ref{vgsolngeneral}) means that bounding higher order derivatives directly is a challenging problem; see \cite{gaunt thesis}, Section 3 and Appendix D.   

This application clearly demonstrates the power of our method: bounding just the solution and its first derivative, and then applying Lemma \ref{lemma2.3} yields bounds on derivatives of the solution of any order.  This makes the problem much more tractable: the problem is essentially reduced to bounding the solution and its first derivative.  We end this section by obtaining such bounds; the bounds of Section 3.1.6 then follow immediately from Lemma \ref{lemma2.3} with inequalities (\ref{vgsolnunibound}) and (\ref{vgsolnunibound1}) as base cases. 

By an application of Stirling's inequality, we see that our bounds for $n$-th order derivatives of the solution of $\mathrm{VG}(r,\theta,\sigma,0)$ Stein equation are of order $(r+n)^{-1/2}$ as $r+n\rightarrow\infty$.  This is slower than the $O(r+n)^{-1}$ rate we obtained for the $\theta=0$ case.  It is an open problem to improve the bounds for the $\theta\not=0$ case to this rate.  

Finally, we note an application of our estimates for the solution of the $\mathrm{VG}(r,\theta,\sigma,0)$ Stein equation to the work of \cite{eichelsbacher} on variance-gamma approximation on Wiener space.  In that work, it is shown that a sequence of distributions of random variables in the second Wiener chaos converge to a variance-gamma distribution if and only if their second to sixth order moments converge to those of a variance-gamma random variable.  Bounds on the rate of convergence were also derived using the variance-gamma Stein equation.  These bounds were given in terms of an absolute constant, because explicit bounds were not available for derivatives of the solution of the Stein equation.  With our explicit bounds on the derivatives, it is possible to give explicit constants in the bounds of \cite{eichelsbacher}.

\vspace{3mm}

\noindent\emph{Proof of inequalities (\ref{vgsolnunibound}) and (\ref{vgsolnunibound1}).} In order to simplify the calculations, we make the following change of parameters
\begin{equation} \label{parameter} \nu=\frac{r-1}{2}, \qquad \alpha =\frac{\sqrt{\theta^2 +  \sigma^2}}{\sigma^2}, \qquad \beta =\frac{\theta}{\sigma^2}.
\end{equation}
With these parameters, the solution (\ref{vgsolngeneral}) can be written as
\begin{align} \label{ink} f(x) &=-\frac{e^{-\beta x} K_{\nu}(\alpha|x|)}{\sigma^2|x|^{\nu}} \int_0^x e^{\beta y} |y|^{\nu} I_{\nu}(\alpha|y|) \tilde{h}(y) \,\mathrm{d}y \nonumber \\
&\quad-\frac{e^{-\beta x} I_{\nu}(\alpha|x|)}{\sigma^2|x|^{\nu}} \int_x^{\infty} e^{\beta y} |y|^{\nu} K_{\nu}(\alpha|y|)\tilde{h}(y)\,\mathrm{d}y \\
&=-\frac{e^{-\beta x} K_{\nu}(\alpha|x|)}{\sigma^2|x|^{\nu}} \int_0^x e^{\beta y} |y|^{\nu} I_{\nu}(\alpha|y|) \tilde{h}(y) \,\mathrm{d}y \nonumber \\
\label{pen}&\quad+\frac{e^{-\beta x} I_{\nu}(\alpha|x|)}{\sigma^2|x|^{\nu}} \int_{-\infty}^{x} e^{\beta y} |y|^{\nu} K_{\nu}(\alpha|y|)\tilde{h}(y)\,\mathrm{d}y.
\end{align}
The equality between (\ref{ink}) and (\ref{pen}) means that in order to obtain uniform bounds for all $x\in\mathbb{R}$ it suffices to bound the solution and its derivative in the region $x\geq0$ as long as we consider the case of both negative and positive $\beta$.  That is, the uniform bound we derive for $x\geq 0$ also holds for $x\leq0$, and thus all $x\in\mathbb{R}$.  Straightforward calculations show that, for $x\geq 0$,
\begin{align}|f(x)| &\leq\frac{\|\tilde{h}\|}{\sigma^2}\bigg|\frac{e^{-\beta x} K_{\nu}(\alpha x)}{x^{\nu}}\int_0^{x} e^{\beta y}y^{\nu} I_{\nu}(\alpha y)\,\mathrm{d}y \bigg| +\frac{\|\tilde{h}\|}{\sigma^2}\bigg|\frac{e^{-\beta x}I_{\nu}(\alpha x)}{x^{\nu}} \int_x^{\infty} e^{\beta y}y^{\nu} K_{\nu}(\alpha y)\,\mathrm{d}y\bigg|, \nonumber \\
|f'(x)| &\leq\frac{\|\tilde{h}\|}{\sigma^2}\bigg|\frac{\mathrm{d}}{\mathrm{d}x}\bigg(\frac{e^{-\beta x}K_{\nu}(\alpha x)}{x^{\nu}}\bigg)\int_0^x e^{\beta y}y^{\nu} I_{\nu}(\alpha y)\,\mathrm{d}y \bigg| \nonumber\\
&\quad +\frac{\|\tilde{h}\|}{\sigma^2}\bigg|\frac{\mathrm{d}}{\mathrm{d}x}\bigg(\frac{e^{-\beta x}I_{\nu }(\alpha x)}{x^{\nu}}\bigg) \int_x^{\infty} e^{\beta y}y^{\nu} K_{\nu}(\alpha y)\,\mathrm{d}y\bigg|.  \nonumber 
\end{align}

We now note the following bounds, which follow almost immediately from Theorems 3.1 and 3.2 of \cite{gaunt}.  Fix $\nu>-\frac{1}{2}$ and $0<|\beta|<\alpha$.  Then, for all $x>0$,
\begin{eqnarray*}\bigg|\frac{e^{-\beta x} K_{\nu}(\alpha x)}{x^{\nu}}\int_0^{x} e^{\beta y}y^{\nu} I_{\nu}(\alpha y)\,\mathrm{d}y \bigg| &<& \frac{2}{\alpha(2\nu+1)}, \\
\bigg|\frac{e^{-\beta x}I_{\nu}(\alpha x)}{x^{\nu}} \int_x^{\infty} e^{\beta y}y^{\nu} K_{\nu}(\alpha y)\,\mathrm{d}y\bigg| &<& \frac{K(\nu,\gamma)}{\alpha}, \\
\bigg|\frac{\mathrm{d}}{\mathrm{d}x}\bigg(\frac{e^{-\beta x}K_{\nu}(\alpha x)}{x^{\nu}}\bigg)\int_0^x e^{\beta y}y^{\nu} I_{\nu}(\alpha y)\,\mathrm{d}y \bigg| &<& \frac{2(\gamma+1)}{2\nu+1}, \\
\bigg|\frac{\mathrm{d}}{\mathrm{d}x}\bigg(\frac{e^{-\beta x}I_{\nu }(\alpha x)}{x^{\nu}}\bigg) \int_x^{\infty} e^{\beta y}y^{\nu} K_{\nu}(\alpha y)\,\mathrm{d}y\bigg|&<& K(\nu,\gamma),
\end{eqnarray*}
where $\gamma=\beta/\alpha$ and
\begin{equation*}K(\nu,\gamma)=\begin{cases} \displaystyle \frac{\sqrt{\pi}}{(1-\gamma^2)^{\nu+1/2}}\frac{\Gamma\left(\nu+\frac{1}{2}\right)}{\Gamma(\nu+1)}, & \:  \nu\geq\frac{1}{2}, \\
\displaystyle \frac{6\Gamma(\nu+\frac{1}{2})}{1-|\gamma|}, & \:   |\nu|<\frac{1}{2}. \end{cases}
\end{equation*}
Applying these inequalities to bound the above expressions for $|f(x)|$ and $|f'(x)|$, and recalling that it suffices to bound the solution for $x\geq0$ leads to the following bounds, which hold for all $x\in\mathbb{R}$,
\begin{eqnarray*}|f(x)|&\leq& \frac{\|\tilde{h}\|}{\sigma^2\alpha}\bigg(\frac{2}{2\nu+1}+K(\nu,\gamma)\bigg), \\
|f'(x)|&\leq& \frac{\|\tilde{h}\|}{\sigma^2}\bigg(\frac{2(\gamma+1)}{2\nu+1}+K(\nu,\gamma)\bigg).
\end{eqnarray*}
Using (\ref{parameter}) to return to the original parameters yields (\ref{vgsolnunibound}) and (\ref{vgsolnunibound1}), as required. \hfill $\square$

\subsubsection{Density proportional to $e^{-x^4/12}$}

We first prove the bounds \eqref{dbounds1} and \eqref{dbounds2}. Note that the density $p$ and its Stein equation \eqref{411} satisfy the general Conditions 3.1 and 3.2 from the paper \cite{dobler beta}. Hence, 
by Proposition 3.13 (a) of that paper, we conclude that 
\begin{equation}\label{cwp1}
 |f(x)|\leq \frac{\int_x^\infty tp(t)\,\mathrm{d}t}{p(x)}\|h'\|
\end{equation}
holds for each $x\in\mathbb{R}$. Writing $\Phi$ for the standard normal distribution function and $\varphi=\Phi'$ for its density, it is an easy matter of partial integration to check that 
\begin{equation}\label{repp}
 p(x)=c_1\sqrt{2\pi}\varphi\bigg(\frac{x^2}{\sqrt{6}}\bigg)\quad\text{and}\quad\int_x^\infty tp(t)\,\mathrm{d}t=c_1\sqrt{3\pi}\bigg(1-\Phi\bigg(\frac{x^2}{\sqrt{6}}\bigg)\bigg),
\end{equation}
where we denote by 
\begin{equation*}
 c_1:=\frac{\sqrt{2}}{3^{1/4}\Gamma(1/4)}
\end{equation*}
the normalizing constant of $p$. Hence, from \eqref{cwp1} we obtain the bound 
\begin{equation}\label{cwp2}
 |f(x)|\leq\frac{\sqrt{6}}{2}\;\frac{1-\Phi\bigl(\frac{x^2}{\sqrt{6}}\bigr)}{\varphi\bigl(\frac{x^2}{\sqrt{6}}\bigr)}\|h'\|.
\end{equation}
The Gaussian Mill's ratio inequality states that, for each $t>0$, 
\begin{equation}\label{GMill}
 \frac{t}{1+t^2}\leq\frac{1-\Phi(t)}{\varphi(t)}\leq\min\bigg(\sqrt{\frac{\pi}{2}},\,\frac{1}{t}\bigg).
\end{equation}
Hence, from \eqref{cwp2} and \eqref{GMill} we conclude that 
\begin{equation}\label{cwp3}
 |f(x)|\leq\min\bigg(\frac{\sqrt{3\pi}}{2},\,\frac{3}{x^2}\bigg)\|h'\|,
\end{equation}
which already gives the first bound in \eqref{dbounds1} and \eqref{dbounds2}, respectively. For the bounds on $f'$ we use the once iterated Stein equation \eqref{cwi1} together with the first bound in \eqref{elbounds} and 
\eqref{cwp3} to obtain
\begin{align}
 |f'(x)|&\leq\min\left(\frac{3^{1/4}\Gamma(1/4)}{2\sqrt{2}},\,\frac{3}{|x|^3}\right)|h'(x)+x^2f(x)|\notag\\
 &\leq\min\left(\frac{3^{1/4}\Gamma(1/4)}{2\sqrt{2}},\,\frac{3}{|x|^3}\right)\|h'\|+\min\left(\frac{3^{1/4}\Gamma(1/4)}{2\sqrt{2}},\,\frac{3}{|x|^3}\right)|x^2f(x)|\notag\\
 &\leq\min\left(\frac{3^{1/4}\Gamma(1/4)}{2\sqrt{2}},\,\frac{3}{|x|^3}\right)\|h'\|+3\min\left(\frac{3^{1/4}\Gamma(1/4)}{2\sqrt{2}},\,\frac{3}{|x|^3}\right)\|h'\|\notag\\
 &=\min\left(2^{1/2}3^{1/4}\Gamma(1/4),\,\frac{12}{|x|^3}\right)\|h'\|,
\end{align}
as claimed. In order to prove the bound on $f''$, we rely on the following representation of the second derivative of the Stein equation within the so-called density approach of Stein's method:
\begin{equation}\label{repf2}
 f''(x)=h'(x)+\frac{U(x)}{p(x)}\int_{-\infty}^xF(s)h'(s)\,\mathrm{d}s+\frac{V(x)}{p(x)}\int_x^\infty\bigl(1-F(s)\bigr)h'(s)\,\mathrm{d}s,
\end{equation}
where $F$ is the distribution function corresponding to $p$ and 
\begin{eqnarray*}
 U(x)&=&-\bigl(1-F(x)\bigr)\bigg(x^2+\frac{x^6}{9}\bigg)+\frac{x^3}{3}p(x),\\
 V(x)&=&-F(x)\bigg(x^2+\frac{x^6}{9}\bigg)-\frac{x^3}{3}p(x).
\end{eqnarray*}
Identity \eqref{repf2} is fairly standard in Stein's method. A proof can, for instance, be found in the arXiv version of the paper \cite{doblersrw}. 
With the help of the Mill's ratio type inequalities (5.51)--(5.54) for the distribution corresponding to $p$ from \cite{el10}, one can easily see that both $U$ and $V$ are non-positive functions. 
Using the fact that 
\begin{align*}
 \int_{-\infty}^xF(s)\,\mathrm{d}s=xF(x)-\int_{-\infty}^xsp(s)\,\mathrm{d}s=xF(x)+\int_x^\infty sp(s)\,\mathrm{d}s
 \end{align*}
and
\begin{align*}
 \int_x^\infty\bigl(1-F(s)\bigr)\,\mathrm{d}s=-x\bigl(1-F(x)\bigr)+\int_x^\infty sp(s)\,\mathrm{d}s
\end{align*}
we conclude that 
\begin{align*}
|f''(x)|&\leq \|h'\|+\|h'\| \left(\frac{-U(x)}{p(x)}\int_{-\infty}{x}F(s)\,\mathrm{d}s+\frac{-V(x)}{p(x)}\int_x^\infty\bigl(1-F(s)\bigr)\,\mathrm{d}s\right)\\
 &=\|h'\|+\|h'\|\bigg(\bigg(-\frac{x^3}{3}+\frac{\bigl(1-F(x)\bigr)\bigl(x^2+\frac{x^6}{9}\bigr)}{p(x)}\bigg)\bigg(xF(x)+\int_x^\infty sp(s)\,\mathrm{d}s\bigg)\\
 &\quad+\bigg(\frac{x^3}{3}+\frac{F(x)\big(x^2+\frac{x^6}{9}\big)}{p(x)}\bigg)\bigg(-x\big(1-F(x)\big)+\int_x^\infty sp(s)\,\mathrm{d}s\bigg)\\
 &=\|h'\|+\|h'\|\left(-\frac{x^4}{3}+\bigg(x^2+\frac{x^6}{9}\bigg)\frac{\int_x^\infty sp(s)\,\mathrm{d}s}{p(x)}\right)\\
 &\leq\|h'\|+\|h'\|\left(-\frac{x^4}{3}+\bigg(x^2+\frac{x^6}{9}\bigg)\frac{3}{x^2}\right)\\
 &=4\|h'\|,
\end{align*}
where we used \eqref{repp} and \eqref{GMill}, again, to obtain the final inequality. Using the second bound in \eqref{elbounds} for the iterated Stein equation \eqref{cwi1} as well as \eqref{cwp3}, we easily obtain
\begin{align*}
 |f''(x)|&\leq 2|h'(x)+x^2f(x)|\leq2\|h'\|+2\cdot 3\|h'\|=8\|h'\|,
\end{align*}
which is a very simple derivation as compared to that of the previous bound.

Next, we prove \eqref{genelbound1}, \eqref{inian} and \eqref{recan} by induction on $n$. We omit the cases $n=0,1,2$ and assume that $n\geq3$ and that the claim holds for all $0\leq k<m$. It is easily checked that 
\begin{equation}\label{prefac}
 |x|\min\bigg(\frac{1}{2c_1},\,\frac{3}{|x|^3}\bigg)\leq \frac{3}{\bigl(6c_1\bigr)^{2/3}}\quad\text{and}\quad x^2\min\bigg(\frac{1}{2c_1},\,\frac{3}{|x|^3}\bigg)\leq \frac{3}{\bigl(6c_1\bigr)^{1/3}}.
\end{equation}
Hence, from \eqref{elbounds} applied to $f_n$, from the definition of $h_n$ and the induction hypothesis we conclude that
\begin{align*}
 |f^{(n)}(x)|&\leq \min\bigg(\frac{1}{2c_1},\,\frac{3}{|x|^3}\bigg)\|h_n\|\\
 &\leq  \min\bigg(\frac{1}{2c_1},\,\frac{3}{|x|^3}\bigg)\bigg(\|h^{(n)}\| +\frac{n(n-1)(n-2)}{6c_1}\sum_{j=0}^{n-3}a^{(n-3)}_{n-3-j}\|h^{(n-3-j)}\|\\
&\quad+n(n-1)\frac{3}{\bigl(6c_1\bigr)^{2/3}}\sum_{j=0}^{n-2}a^{(n-2)}_{n-2-j}\|h^{(n-2-j)}\|+\frac{3n}{\bigl(6c_1\bigr)^{1/3}}\sum_{j=0}^{n-1}a^{(n-1)}_{n-1-j}\|h^{(n-1-j)}\|\bigg).
\end{align*}
Now, collecting terms and using the induction hypothesis on \eqref{inian} we obtain
\begin{align*}
|f^{(n)}(x)|&\leq \min\bigg(\frac{1}{2c_1},\,\frac{3}{|x|^3}\bigg)\Biggl(\|h^{(n)}\|+\frac{3n}{\bigl(6c_1\bigr)^{1/3}}\|h^{(n-1)}\|\\
&\quad+\left(\frac{3n}{\bigl(6c_1\bigr)^{1/3}}a^{(n-1)}_{n-2}+\frac{3n(n-1)}{\bigl(6c_1\bigr)^{2/3}}a_{n-2}^{(n-2)}\right)\|h^{(n-2)}\|\\
&\quad+\sum_{k=3}^n\biggl(\frac{3n}{\bigl(6c_1\bigr)^{1/3}}a_{n-k}^{(n-1)}+\frac{3n(n-1)}{\bigl(6c_1\bigr)^{2/3}}a_{n-k}^{(n-2)}+\frac{n(n-1)(n-2)}{6c_1}a_{n-k}^{(n-3)}\biggr)\|h^{(n-k)}\|\Biggr)\\
&=\min\bigg(\frac{1}{2c_1},\,\frac{3}{|x|^3}\bigg)\Biggl(\|h^{(n)}\|+\frac{3n}{\bigl(6c_1\bigr)^{1/3}}\|h^{(n-1)}\| +\frac{12n(n-1)}{\bigl(6c_1\bigr)^{2/3}}\|h^{(n-2)}\|\\
&\quad+\sum_{k=3}^n\biggl(\frac{3n}{\bigl(6c_1\bigr)^{1/3}}a_{n-k}^{(n-1)}+\frac{3n(n-1)}{\bigl(6c_1\bigr)^{2/3}}a_{n-k}^{(n-2)}+\frac{n(n-1)(n-2)}{6c_1}a_{n-k}^{(n-3)}\biggr)\|h^{(n-k)}\|\Biggr).
\end{align*}
Note that, in passing, we have also proved \eqref{hnbound} and, now, \eqref{elbounds} yields \eqref{genelbound2}. 

\subsubsection{Multivariate normal distribution}

The iterative method, which has been presented mostly for linear ODEs in this paper, applies equally to linear PDEs.  Indeed, we can obtain bounds for $n$-th order partial derivatives of the solution of the multivariate normal Stein equation if we are given just the bound (\ref{mvnbound}).  For simplicity, we consider the standard multivariate normal distribution with covariance matrix $\Sigma=I_d$, the $d\times d$ identity matrix.  Recall that, for $k\geq1$,
\begin{equation*}L\bigg(\frac{\partial^k f(x)}{\prod_{j=1}^k\partial x_{i_j}}\bigg)=\frac{\partial^kh(x)}{\prod_{j=1}^k\partial x_{i_j}}+k\frac{\partial^kf(x)}{\prod_{j=1}^k\partial x_{i_j}}.
\end{equation*}
Then applying the iterative technique with the bound (\ref{mvnbound}) gives
\begin{align*}\bigg\|\frac{\partial^2 f}{\partial x_i\partial x_j}\bigg\|&\leq\sqrt{\frac{\pi}{2}}\min_{l\in\{i,j\}}\bigg\|\frac{\partial h(x)}{\partial x_l}+\frac{\partial h(x)}{\partial x_l}\bigg\|\leq\sqrt{\frac{\pi}{2}}\min_{l\in\{i,j\}}\bigg\{\bigg\|\frac{\partial h}{\partial x_l}\bigg\|+\bigg\|\frac{\partial f}{\partial x_l}\bigg\|\bigg\}\\
& \leq \sqrt{\frac{\pi}{2}}\min_{l\in\{i,j\}}\bigg\|\frac{\partial h}{\partial x_l}\bigg\|+\frac{\pi}{2}\|\tilde{h}\|.
\end{align*}
Repeating this procedure allows us to bound partial derivatives of $f$ of any order.  However, the resulting bound would be much worse than the best bounds from the literature.

\section{Connection between Stein equations and Poisson equations}

If $L$ in equation \eqref{steineqn} takes the form of a second order (possibly degenerate)
elliptic differential operator, then the Stein equation coincides with
the Poisson equation
\begin{equation}
Lf=h-\mu(h).\label{eq:poisson}
\end{equation}
There is a large literature on existence and regularity results for
elliptic differential equations, see \cite{evans2010} and \cite{GilbargTrudinger}
as well as references thereof and therein. Most of these results focus
on bounded domains $D\subsetneq\mathbb{R}^{d}.$ However, \cite{pardoux}
establishes regularity results for $D=\mathbb{R}^{d}.$ We concentrate
on presenting the formal relation before summarising and exploiting
the results of \cite{pardoux} in Section \ref{sub:ResultsPardoux}.

These regularity results are achieved by studying the SDE
\begin{equation}
\mathrm{d}X_{t}=b(X_{t})\,\mathrm{d}t+\sigma(X_{t})\,\mathrm{d}B_{t},\label{eq:diffusion}
\end{equation}
where $B_t$ is $d_1$-dimensional Brownian motion, that has 
\begin{equation*}
L\phi(x)=\sum_{i=1}^{d}b_{i}(x)\frac{\partial\phi}{\partial x_{i}}+\frac{1}{2}\sum_{i,j=1}^{d}a_{ij}(x)\frac{\partial^{2}\phi}{\partial x_{i}\partial x_{j}},\quad a(x)=\sigma\sigma^{*}(x)\label{eq:generator}
\end{equation*}
as its infinitesimal generator. In the Stein equation literature this
is known as generator approach which was introduced by Barbour
and G\"{o}tze \cite{barbour2}, \cite{gotze}, but instead of looking
at a case by case basis we establish a general connection that has
to best of our knowledge not been exploited in the literature. However, this connection has been hinted at in Section 6.2 of \cite{2010MattinglyPoisson} without making the connection with regularity results of \cite{pardoux}. The iteration of the regularity results of \cite{pardoux} similar to results presented below has been used in \cite{teh2015sgldB} to analyse MCMC algorithms with subsampling. 
Recently, in \cite{Gurvich2014} and in \cite{DaiBrav}, the generator approach of Stein's method for certain classes of diffusion processes has been used for applications related to queuing theory. In fact,
\cite{Gurvich2014} seems to re-develop this technique, being unaware of Barbour's and G\"otze's previous insights. 

Stein's method can be viewed as a quantitative version of the Echeverr\'{i}­a-Weiss
theorem (more details about this theorem can be found in Sections
4.1-4.2 in \cite{lamberton2002recursive} and in \cite{kurtz}). 
\begin{theorem}
(Echeverr\'{i}­a-Weiss theorem)\label{prop:EcheverriSpecial} Suppose
there is a probability density $\mu$ on $\mathbb{R}^{d}$ such that,
for all $\varphi\in C_{c}^{2}$, 
\[
\int_{\mathbb{R}^{d}}L\varphi(x)\mu(x)\,\mathrm{d}x=0.
\]
Then $\mu$ is an invariant distribution for the SDE in equation (\ref{eq:diffusion}). 
\end{theorem}
This theorem illustrates the relation between the invariant/stationary distribution
of the associated stochastic process and the Stein differential operator.
Moreover, sufficiently fast convergence of the stochastic process
to its invariant distribution guarantees the existence of a solution
to the Poisson equation (\ref{eq:poisson}). 

\subsection{Solution of the Poisson equation}\label{sub:ResultsPardoux}

Theorem 1 and 2 in \cite{pardoux} characterise the
smoothness and growth of the solution to the Poisson equation associated
with equation (\ref{eq:diffusion}).  We now state a simplified version of these theorems, which we make use of in Section 4.2.  Before doing so, we state some assumptions.  We suppose that $b$ is a locally bounded Borel vector function of dimension $d_1\geq d$, that $\sigma$ is a $d\times d_1$ matrix-valued uniformly continuous function, and that $\sigma\sigma^*$ is bounded and positive-definite.  We also suppose that 
\begin{equation}\left\langle b(x),x/|x|\right\rangle   \leq -r|x|,\quad |x|\ge M_{0},\label{eq:assPardouxDrift}
\end{equation}
with $M_0\geq0$ and $r>0$.  This condition prevents the solution of the SDE (\ref{eq:diffusion})
from blowing up, so that the process $\{X_t\}$
 is well-defined for all $t > 0$.

\begin{theorem}
\label{thm:pardouxPoisson}\cite{pardoux}\  Suppose that the above assumptions hold, and that there also exist $C>0$ and $\beta\geq0$ such that
\begin{equation}
\left|h(x)-\mu(h)\right|\leq C\left(1+|x|\right)^{\beta}.\label{eq:poissonhbound1}
\end{equation}
Then there exists a solution $f$, belonging to the Sobolev class $W_{p,\mathrm{loc}}^{2}$ for any $p>1$, to the Poisson equation (\ref{eq:poisson}).

\begin{enumerate}
\item If $\beta<0$, then $f$ is bounded. Moreover, 
\[
\sup_x\left|f(x)\right|\leq C^{\prime}(\beta)\sup_x\big|(h(x)-\mu(h))\left(1+|x|\right)^{-\beta}\big|,
\]
and there exists a constant $ C^{\prime\prime}(\beta)$ such that
\[
|\nabla f(x)|\leq C^{\prime\prime}(\beta)\big(1+\left|x\right|^{\beta}\big).
\]

\item Suppose now that $\beta>0$.  Then there exists $C^{\prime}(\beta)$ such that 
\[
|f(x)|\leq C^{\prime}(\beta)\left(1+|x|\right)^{\beta}.
\]
Finally, there exists $ C^{\prime\prime}(\beta)$ such that 
\[
|\nabla f(x)|\leq C^{\prime\prime}(\beta)\left(1+|x|^{\beta}\right).
\]

\end{enumerate}
In both cases $C^{\prime}(\beta)$, $C^{\prime\prime}(\beta) \rightarrow\infty$ as $\left|\beta\right|\rightarrow0.$
\end{theorem}

Recall that $P_{t}f(x)=\mathbb{E}_{x}f(X_{t})$ is the semigroup associated
with $L$, that can also be expressed by $\int f(y)\mu_{t}^{x}(\mathrm{d}y)$
using the distribution $\mu_{t}^{x}$ of $X_{t}$ started at $x$.
On that basis we see that 
\begin{equation}
f(x)=-\int_{0}^{\infty}\left(P_{t}(h-\mu(h))\right)(x)\,\mathrm{d}t\label{eq:formalpoisson}
\end{equation}
is a formal solution to Poisson equation (\ref{eq:poisson}) by calculating
\begin{align*}
Lf(x) & =  -\int_{0}^{\infty}L\left(P_{t}(h-\mu(h))\right)(x)\,\mathrm{d}t=-\int_{0}^{\infty}\frac{\mathrm{d}}{\mathrm{d}t}\left(P_{t}(h-\mu(h))\right)(x)\,\mathrm{d}t\\
 & =  h(x)-\mu(h)-\lim_{t\rightarrow\infty}\left(P_{t}(h-\mu(h)\right)(x),
\end{align*}
where we use that the generator satisfies $LP_{t}=\frac{\mathrm{d}}{\mathrm{d}t}P_{t}$.
Thus, the existence of $f$ can be shown by establishing that
$\lim_{s\rightarrow\infty}-\int_{0}^{s}\left(P_{t}(h-\mu(h))\right)(x)\,\mathrm{d}t$
exists in an appropriate functions space and solves the Poisson equation
(\ref{eq:poisson}) in a weak sense. 

The regularity results comes from classical regularity results for
elliptic partial differential equation, more precisely from the following
inequality
\begin{equation}
\left|\nabla f(x)\right|\leq C\left(\left|f\right|_{L^{\infty}\left(B_{x},2\right)}+\left|Lf\right|_{L^{\infty}\left(B_{x},2\right)}\right),\label{eq:poisRegBound}
\end{equation}
where $B_{x,r}$ is a ball of radius $r$ around $x$; see p$.$ 1070
of \cite{pardoux}. The second term on the RHS is
just the RHS of the Poisson equation which is bounded by assumption.
The problem lies in the term $\left|f\right|_{L^{\infty}\left(B_{x},2\right)}.$
It is known that for typical classes of test functions $h$, like Lipschitz functions, the solution $f$ is not uniformly bounded on $\mathbb{R}^d$. 
However, such uniform bounds on $f$ would would be necessary in order to apply the iterative procedure. Nevertheless, in the next section we obtain bounds of the form
\[
\left|\partial^{\alpha}f\right|\leq C\bigg(1+|x|^{\phi(\left|\alpha\right|)}+\sum_{i=0}^{\left|\alpha\right|}\left|\partial^{\beta_{i}}\left(h-\mu(h)\right)\right|\bigg),
\]
where we used multi-index notation for derivatives and $\left|\beta_{i}\right|=i$
and $\beta_{0}<\cdots<\beta_{\left|\alpha\right|}.$ This discussion
highlights that a bound on $\left|\nabla f\right|$ involving $\left|f\right|_{L^{\infty}\left(B_{x},2\right)}$
is too weak to obtain results like the classical plug-in theorems which are familiar from Stein's method for univariate distributions. However, even non-uniform bounds on $f$ with a polynomial order of growth can 
sometimes be of use, if one has control of certain moments of the random variable under consideration. This has been exemplified recently in \cite{Gurvich2014}, \cite{DaiBrav} and in \cite{gaunt normal}.

\begin{remark}
\emph{Even though the use of (\ref{eq:poisRegBound}) is fundamentally flawed,
we describe how bounds on $\left|f\right|_{L^{\infty}\left(B_{x},2\right)}$
are obtained in \cite{pardoux}}.\end{remark}
\begin{enumerate}
\item \textbf{a priori using a direct calculation following page 1068 of
\cite{pardoux} }

Using equation (\ref{eq:formalpoisson}), bounds can be obtained as
follows 
\begin{align*}
\left|f(x)\right| & \leq  \int_{0}^{\infty}\left|P_{t}(h-\mu(h))(x)\right|\,\mathrm{d}t\\
 & =  \int_{0}^{\infty}\|h\|\left|\int\frac{h(y)}{\|h\|}(\mu_{t}^{x}-\mu)(\mathrm{d}y)\right|\,\mathrm{d}t\\
 & \leq  \|h\|\int_{0}^{\infty}\left|\mu_{t}^{x}-\mu\right|_{TV}\,\mathrm{d}t.
\end{align*}
For most processes of interest the convergence rate of $\left|\mu_{t}^{x}-\mu\right|_{TV}$
depends on $x$. In case of the Langevin equation, we study below,
the bounds would be of the form 
\[
\left|P_{t}-\mu\right|_{TV}\leq CV(x)\exp\left(-\lambda t\right);
\]
see Theorem 6.1 of \cite{Meyn1993continuous3}. 
Typically, $1\le V(x)$ grows at least
linearly in $|x|$. This is not a deficiency of the method
because similar bounds are strict for the Ornstein-Uhlenbeck process. While the convergence to equilibrium thus plays a role in bounding $|f|$ it is not so clear how the convergence relates to bounds on derivatives of $f$.

\item \textbf{a posteriori using a representation} \textbf{formula }

The result in Theorem \ref{thm:pardouxPoisson} is proved using the
representation formula 
\[
f(x)=\mathbb{E}_{x}f\left(X_{\tau^{R}}\right)+\mathbb{E}_{x}\int_{0}^{\tau^{R}}\left(h(X_{t})-\mu(h)\right)\,\mathrm{d}t
\]
in combination with It\^{o}'s formula and (\ref{eq:poissonhbound1}). For more details see p$.$ 1070 of
\cite{pardoux}.

\end{enumerate}

\subsection{The over-damped Langevin SDE\label{sub:IteratingRegularityPoissonEquations}}

We consider the Langevin SDE in $\mathbb{R}^{d}$,

\begin{equation*}
\mathrm{d}X_{t}=\frac{1}{2}\nabla\log{\mu(X_{t})}\,\mathrm{d}t+\mathrm{d}W_{t},\label{eq:Langevin}
\end{equation*}
we see that $\mu$ is the invariant distribution by verifying that
it solves the corresponding Fokker-Planck equation ($L^{\star}\mu=0$,
where $L^{\star}$ is the adjoint operator to $L$). Matching the
coefficients to equation (\ref{eq:diffusion}), we see that 
\begin{equation*}
b(x)  =  \frac{1}{2}\nabla\log\mu(x)\quad \text{and} \quad
\sigma  =  I.
\end{equation*}
In order for Stein's method to be applicable, we need bounds on
higher derivatives as well. For this reason, we will iterate Theorem
\ref{thm:pardouxPoisson}. In order to do so, we note that the derivatives of $f$
can be expressed as solutions to Poisson equations with different RHSs:
\begin{eqnarray}
Lf & = & h-\mu(h),\label{eq:PoiDeriv0}\\
L\partial_{i}f & = & \partial_{i}h-\frac{1}{2}\nabla f\cdot\partial_{i}b,\label{eq:PoiDeriv1}\\
L\partial_{ij}f & = & \partial_{ij}h-\frac{1}{2}\nabla\partial_{j}f\cdot\partial_{i}b-\frac{1}{2}\nabla f\cdot\partial_{ij}b-\frac{1}{2}\nabla f\cdot\partial_{j}b,\label{eq:PoiDeriv2}\\
L\partial_{ijk}f & = & \partial_{ijk}h-\frac{1}{2}\nabla\partial_{jk}f\cdot\partial_{i}b-\frac{1}{2}\nabla\partial_{j}f\cdot\partial_{ik}b-\frac{1}{2}\nabla\partial_{k}f\cdot\partial_{ij}b\nonumber \\
 &  & -\frac{1}{2}\nabla f\cdot\partial_{ijk}b-\frac{1}{2}\nabla\partial_{k}f\cdot\partial_{jk}b-\frac{1}{2}\nabla f\cdot\partial_{jk}b-\frac{1}{2}\nabla f\cdot\partial_{k}b. \label{eq:PoiDeriv3}
\end{eqnarray}
We will denote by $\beta_{f,i}$ numbers that satisfy 
\begin{equation}
\sup_{\left|\alpha\right|=i}|\partial^{\alpha}f(x)|\lesssim\big(1+|x|^{\beta_{f,i}}\big),\label{eq:betanotation}
\end{equation}
where we used multi-index notation for derivatives. We use a similar
notation for the derivatives of $f$, that is $\beta_{f,i}$ and assume
that these bounds are a priori given.

Using Theorem \ref{thm:pardouxPoisson} we can obtain $\beta_{f,i}$ to
satisfy (\ref{eq:betanotation}) in terms of the $\beta$'s,
which we formulate as the following proposition.
\begin{proposition}
Suppose that $h$ and its derivatives up to third order are bounded, and that assumption (\ref{eq:assPardouxDrift})
holds. Then, for any $\epsilon>0$, 
\begin{align*}&\beta_{f,0}  =  \epsilon, \quad \beta_{f,1}  =  \epsilon, \quad \beta_{f,2}  =  \epsilon\vee\frac{\beta_{b,1}}{2}, \quad \beta_{f,3}  =  \epsilon\vee\beta_{b,1}\vee\frac{\beta_{b,2}}{2} \\
&\text{and}\quad \beta_{f,4}  =  \epsilon\vee\frac{1}{2}\left(3\beta_{b,1}\vee\left(\beta_{b,1}+\beta_{b,2}\right)\vee\beta_{b,3}\right)
\end{align*}
satisfy equation (\ref{eq:betanotation}).\end{proposition}
\begin{proof}
Applying Theorem \ref{thm:pardouxPoisson} to equation (\ref{eq:PoiDeriv0})
implies that $\beta_{f,0}:=\beta_{f,1}:=\epsilon$ satisfies equation
(\ref{eq:betanotation}).  Applying Theorem \ref{thm:pardouxPoisson} to equation (\ref{eq:PoiDeriv1}) yields that $\beta_{f,2}\leq\epsilon\vee\beta_{b,1}$, and then applying the theorem to equation (\ref{eq:PoiDeriv2}) gives $\beta_{f,3}\leq2\beta_{b,1}\vee\beta_{b,2}$.  Finally, applying Theorem \ref{thm:pardouxPoisson} to equation (\ref{eq:PoiDeriv3}) yields that 
\begin{align*}
\beta_{f,4} & \leq  \left(\beta_{b,1}+\left(2\beta_{b,1}\vee\beta_{b,2}\right)\right)\vee\left(\beta_{b,1}+\beta_{b,2}\right)\vee\beta_{b,3}\\
 & \leq  3\beta_{b,1}\vee\left(\beta_{b,1}+\beta_{b,2}\right)\vee\beta_{b,3},
\end{align*}
completing the proof.
\end{proof}

\subsection*{Acknowledgements}
RG is supported by a Dame Kathleen Ollerenshaw Fellowship and acknowledges support from EPSRC grant EP/K032402/1. SV's research is partially supported by the grants EP/K009850/1 and EP/N000188/1.  We would like to thank the referees for a very thorough reading of our paper, and for their helpful suggestions.

\appendix

\normalsize

\section{Further calculations}

For the sake of brevity, the following formulas were given in Section 3 without proof.  We provide their simple proofs here.

\subsection{Normal distribution}

We now prove the following inequality that is found in Section 3.2:
\begin{equation*}\|f^{(n)}\|\leq\sum_{j=0}^n\bigg(\frac{\pi}{2}\bigg)^{(n-j+1)/2}\frac{n!}{j!}\|\tilde{h}^{(j)}\|.
\end{equation*}

\begin{proof}Apply Lemma 2.3 with $a_k=k+1$ and $C_k=\sqrt{\frac{\pi}{2}}$ to get
\begin{align*}\|f^{(n)}\|&\leq\sqrt{\frac{\pi}{2}}\sum_{j=0}^n\bigg(\prod_{i=j}^{n-1}\sqrt{\frac{\pi}{2}}(i+1)\bigg)\|\tilde{h}^{(j)}\|\\
&=\sum_{j=0}^n\bigg(\frac{\pi}{2}\bigg)^{(n-1+j)/2}\bigg(\prod_{i=j+1}^ni\bigg)\|\tilde{h}^{(j)}\|\\
&=\sum_{j=0}^n\bigg(\frac{\pi}{2}\bigg)^{(n-j+1)/2}\frac{n!}{j!}\|\tilde{h}^{(j)}\|.
\end{align*}
\end{proof}

\subsection{Beta distribution}

\begin{equation}\label{3441}L_{\alpha,\beta}f(x)=x(1-x)f'(x)+(\alpha-(\alpha+\beta)x)f(x)=\tilde{h}(x)
\end{equation}
and, for $k\geq1$,
\begin{equation}\label{betasi1}L_{\alpha+k,\beta+k}f^{(k)}(x)=h^{(k)}(x)+k(\alpha+\beta)f^{(k-1)}(x).
\end{equation}

\begin{proof}We prove the result be induction.  Differentiating both sides of (\ref{betasi1}) gives
\begin{align*}&x(1-x)f^{(k+2)}(x)+(\alpha+k-(\alpha+\beta+2k)x)f^{(k+1)}(x)+(1-2x)f^{(k+1)}(x)\\
&\quad-(\alpha+\beta+2k)f^{(k)}(x)=h^{(k+1)}(x)+k(\alpha+\beta)f^{(k)}(x).
\end{align*}
Rearranging gives
\begin{align*}&x(1-x)f^{(k+2)}(x)+(\alpha+k+1-(\alpha\beta+2k+2)x)f^{(k+1)}(x)\\
&=h^{(k)}(x)+[(\alpha+\beta+2k)+k(\alpha+\beta+k-1)]f^{(k-1)}(x),
\end{align*}
and so
\begin{align*}
&x(1-x)f^{(k+2)}(x)+(\alpha+k+1-(\alpha\beta+2k+2)x)f^{(k+1)}(x)\\
&=h^{(k)}(x)+(k+1)(\alpha+\beta+k)f^{(k-1)}(x).
\end{align*}
We recognise the left-hand side as $L_{\alpha+k+1,\beta+k+1}f^{(k+1)}(x)$, and so the result has been proved by induction on $k$.
\end{proof}

\subsection{Student's $t$ distribution}

\begin{equation}\label{3551}L_{d,\delta}f(x)=(\delta^2+x^2)f'(x)-(d-1)xf(x)=\tilde{h}(x)
\end{equation}
and, for $k\geq1$,
\begin{equation}\label{tsteinop1}L_{d-2k,\delta}f^{(k)}(x)=h^{(k)}(x)+k(d-k)f^{(k-1)}(x).
\end{equation}

\begin{proof}Differentiating both sides of (\ref{3551}) gives
\begin{align*}&(\delta^2+x^2)f^{(k+2)}(x)-(d-2k-1)xf^{(k+1)}(x)+2xf^{(k+1)}(x)-(d-2k-1)f^{(k)}(x) \\
&=h^{(k+1)}(x)+k(d-k)f^{(k)}(x),
\end{align*}
which on rearranging gives
\begin{align*}L_{d-2(k+1),\delta}&=h^{(k+1)}(x)+[k(d-k)+(d-2k-1)]f^{(k)}(x) \\
&=h^{(k+1)}(x)+(k+1)(d-k-1)f^{(k)}(x),
\end{align*}
as required.
\end{proof}

\begin{equation*}\|f^{(n)}\|\leq \sum_{j=0}^n\bigg(\frac{\sqrt{\pi}}{2\delta}\bigg)^{n+1-j}\frac{n!}{j!}(d-n)_{n-j}A_j\|\tilde{h}^{(j)}\|,
\end{equation*}
where $A_j=\prod_{i=j}^n\frac{\Gamma(\frac{d}{2}-i)}{\Gamma(\frac{d+1}{2}-i)}$.

\begin{proof}Apply Lemma 2.3 with $a_k=(k+1)(d-k-1)$ and $C_k=\frac{\sqrt{\pi}\Gamma(\frac{d}{2}-k)}{2\delta\Gamma(\frac{d+1}{2}-k)}$ to get
\begin{align*}\|f^{(n)}\|&\leq C_n\sum_{j=0}^n\bigg(\sum_{i=j}^{n-1}C_i(i+1)(d-i-1)\bigg)\|\tilde{h}^{(j)}\|\\
&=\sum_{j=0}^n\bigg(\frac{\sqrt{\pi}}{2\delta}\bigg)^{n+1-j}A_j\bigg(\prod_{i=j+1}^n\bigg)\|\tilde{h}^{(j)}\|\\
&=\sum_{j=0}^n\bigg(\frac{\sqrt{\pi}}{2\delta}\bigg)^{n+1-j}\frac{n!}{j!}(d-n)_{n-j}A_j\|\tilde{h}^{(j)}\|.
\end{align*}
\end{proof}

\begin{equation*}\|f^{(2k+1)}\|\leq \sum_{j=0}^k\bigg(\frac{2}{\delta^2}\bigg)^{k-j+1}\frac{(2k)!!}{(2j)!!}(d-2k)_{k-j,2}\|\tilde{h}^{(2j)}\|
\end{equation*}
and
\begin{align*}\|f^{(2k)}\|&\leq \sum_{j=1}^k\bigg(\frac{2}{\delta^2}\bigg)^{k-j+1}\frac{(2k-1)!!}{(2j-1)!!}(d-2k+1)_{k-j,2}\|h^{(2j-1)}\|\\
&\quad+\frac{\sqrt{\pi}\Gamma(\frac{d}{2})}{2\delta\Gamma(\frac{d+1}{2})}\bigg(\frac{2}{\delta^2}\bigg)^{k}(2k-1)!!(d-2k+1)_{k,2}\|\tilde{h}\|.
\end{align*}

\begin{proof}Apply Lemma 2.3 with $a_k=(k+1)(d-k-1)$ and $C_k=\frac{2}{\delta^2}$ to get
\begin{align*}\|f^{(2k+1)}\|&\leq\frac{2}{\delta^2}\sum_{j=0}^k\bigg(\prod_{i=j}^{k-1}\frac{2}{\delta^2}(2i+2)(d-2i-2)\bigg)\|\tilde{h}^{(2j})\|\\
&=\sum_{j=0}^k\bigg(\frac{2}{\delta^2}^{k-j+1}\bigg(\prod_{i=j+1}^k(2i)(d-2i)\bigg)\|\tilde{h}^{(2j})\|\\
&=\sum_{j=0}^k\bigg(\frac{2}{\delta^2}\bigg)^{k-j+1}\frac{(2k)!!}{(2j)!!}(d-2k)_{k-j,2}\|\tilde{h}^{(2j)}\|
\end{align*}
and
\begin{align*}\|f^{(2k)}\|&\leq\frac{2}{\delta^2}\sum_{j=1}^k\bigg(\prod_{i=j}^{k-1}\frac{2}{\delta^2}(2i+1)(d-2i-1)\bigg)\|\tilde{h}^{(2j-1)}\|\\
&\quad+\frac{\sqrt{\pi}\Gamma(\frac{d}{2})}{2\delta\Gamma(\frac{d+1}{2})}\|\tilde{h}\|\cdot\prod_{i=1}^k\frac{2}{\delta^2}(2i-1)(d-2i+1)\\
&\leq \sum_{j=1}^k\bigg(\frac{2}{\delta^2}\bigg)^{k-j+1}\frac{(2k-1)!!}{(2j-1)!!}(d-2k+1)_{k-j,2}\|h^{(2j-1)}\|\\
&\quad+\frac{\sqrt{\pi}\Gamma(\frac{d}{2})}{2\delta\Gamma(\frac{d+1}{2})}\bigg(\frac{2}{\delta^2}\bigg)^{k}(2k-1)!!(d-2k+1)_{k,2}\|\tilde{h}\|.
\end{align*}
\end{proof}

\subsection{Inverse-gamma distribution}

\begin{equation}\label{3ii1}L_{\alpha,\beta}f(x)=x^2f'(x)+(\beta-(\alpha-1)x)f(x)=\tilde{h}(x)
\end{equation}
and, for $k\geq1$,
\begin{equation*}L_{\alpha-2k,\beta}f^{(k)}(x)=h^{(k)}(x)+k(\alpha-k)f^{(k-1)}(x).
\end{equation*}

\begin{proof}Differentiating both sides of (\ref{3ii1}) gives
\begin{align*}&x^2f^{(k+2)}(x)+(\beta-(\alpha-2k-1)x)f^{(k+1)}(x)+2xf^{(k+1)}(x)-(\alpha-2k-1)f^{(k)}(x)\\
&=h^{(k+1)}(x)+k(\alpha-k)f^{(k)}(x),
\end{align*}
which on rearranging gives
\begin{align*}L_{\alpha-2(k+1),\beta}f^{(k+1)}(x)&=h^{(k+1)}(x)+[k(\alpha-k)+(\alpha-2k-1)]f^{(k)}(x)\\
&=h^{(k+1)}(x)+(k+1)(\alpha-k-1)f^{(k)}(x),
\end{align*}
as required.
\end{proof}

\begin{equation*}\|f^{(n)}\|\leq \sum_{j=0}^n\frac{n!}{j!}(\alpha-n)_{n-j}B_j\|\tilde{h}^{(j)}\|,
\end{equation*}
where $B_j=\prod_{i=j}^n\frac{\Gamma(\alpha-2i)}{\beta}\left(\frac{e}{\alpha-2i-1}\right)^{\alpha-2i-1}$.

\begin{proof}Applying Lemma 2.3 with $a_k=(k+1)(\alpha-k-1)$ and $C_k=\frac{\Gamma(\alpha-2k)}{\beta}\bigg(\frac{\alpha}{\alpha-2k-1}\bigg)^{\alpha-2k-1}$ gives
\begin{align*}\|f^{(n)}\|&\leq C_k\sum_{j=0}^n\bigg(\prod_{i=j}^{n-1}C_i(i+1)(\alpha-i-1)\bigg)\|\tilde{h}^{(j)}\|\\
&=\sum_{j=0}^nB_j\bigg(\prod_{i=j+1}^ni(\alpha-i)\bigg)\|\tilde{h}^{(j)}\|\\
&=\sum_{j=0}^n\frac{n!}{j!}(\alpha-n)_{n-j}B_j\|\tilde{h}^{(j)}\|.
\end{align*}
\end{proof}

\subsection{PRR distribution}

\begin{equation}\label{3661}L_sf(x)=sf''(x)-xf'(x)-2(s-1)f(x)=\tilde{h}(x)
\end{equation}
and, for $k\geq1$,
\begin{equation*}L_sf^{(k)}(x)=h^{(k)}(x)+kf^{(k)}(x).
\end{equation*}

\begin{proof}Differentiating both sides of (\ref{3661}) gives
\begin{align*}L_sf^{(k+1)}(x)&=sf^{(k+3)}(x)-xf^{(k+2)}(x)-2(s-1)f^{(k+1)}(x)-f^{(k+1)}(x)\\
&=h^{(k+1)}(x)+kf^{(k+1)}(x),
\end{align*}
which on rearranging gives
\begin{align*}L_sf^{(k+1)}(x)=h^{(k+1)}(x)+(k+1)f^{(k+1)}(x),
\end{align*}
as required.
\end{proof}

\begin{equation*}\|f^{(n)}\|\leq (n-1)!\sum_{j=0}^{n-1}\frac{(2\pi)^{(n-j)/2}}{j!}\|h^{(j)}\|.
\end{equation*}

\begin{proof}Apply Lemma 2.4 with $a_k=k+1$ and $C_k=\sqrt{2\pi}$ to get
\begin{align*}\|f^{(n)}\|&\leq\sqrt{2\pi}\sum_{j=0}^{n-1}\bigg(\prod_{i=j}^{n-2}\sqrt{2\pi}(i+1)\bigg)\|h^{(j)}\|\\
&= (n-1)!\sum_{j=0}^{n-1}\frac{(2\pi)^{(n-j)/2}}{j!}\|h^{(j)}\|.
\end{align*}
\end{proof}

\begin{equation*}\|f^{(2k+1)}\|\leq \sqrt{2\pi}C^{k}(2k-1)!!\|h\|+ \sum_{j=1}^{k-1}C^{k-j+1}\frac{(2k-1)!!}{(2j-1)!!}\|h^{(2j-1)}\|
\end{equation*}
and
\begin{equation*}\|f^{(2k)}\|\leq  \sum_{j=0}^{k-1}C^{k-j+1}\frac{(2k)!!}{(2j)!!}\|h^{(2j)}\|,
\end{equation*}
where $C=4$ if $s=1/2$ and $C=2(\sqrt{\pi}s+s^{-1})$ if $s\geq 1$.

\begin{proof}Apply Lemma 2.4 with $a_k=k+1$ and $D_k=C$ to get
\begin{align*}\|f^{(2k+1)}\|&\leq C\sum_{j=1}^k\bigg(\prod_{i=j}^{k-1}C(2i+1)\bigg)\|h^{(2j-1)}\|+\sqrt{2\pi}\|h\|\prod_{i=1}^kC(2i-1)\\
&=\sqrt{2\pi}C^{k}(2k-1)!!\|h\|+ \sum_{j=1}^{k-1}C^{k-j+1}\frac{(2k-1)!!}{(2j-1)!!}\|h^{(2j-1)}\|
\end{align*}
and
\begin{align*}\|f^{(2k)}\|&\leq C\sum_{j=1}^{k-1}\bigg(\prod_{i=j}^{k-2}C(2i+2)\bigg)\|h^{(2j)}\| \\
&=\sum_{j=0}^{k-1}C^{k-j+1}\frac{(2k)!!}{(2j)!!}\|h^{(2j)}\|.
\end{align*}
\end{proof}

\subsection{Variance-gamma distribution}

\begin{equation}\label{3771}L_{r,\theta,\sigma}f(x)=\sigma^2xf''(x)+(\sigma^2r+2\theta x)f'(x)-(r\theta-x)f(x)=\tilde{h}(x)
\end{equation}
and for $k\geq1$, we have
\begin{equation*}L_{r+k,\theta,\sigma}f^{(k)}(x)=h^{(k)}(x)+kf^{(k-1)}(x)+k\theta f^{(k)}(x).
\end{equation*}

\begin{proof}Differentiating both sides of (\ref{3771}) gives
\begin{align*}&\sigma^2xf^{(k+3)}(x)+(\sigma^2(r+k)+2\theta x)f^{(k+2)}(x)+(r+k)\theta-x)f^{(k+1)}(x) \\
&\quad+\sigma^2f^{(k+2)}(x)+2\theta f^{(k+1)}(x)-f^{(k)}(x) \\
&=h^{(k+1)}(x)+kf^{(k)}(x)+k\theta f^{(k+1)}(x),
\end{align*}
which on rearranging gives
\begin{align*}L_{r+k+1,\theta,\sigma}f^{(k+1)}(x)=h^{(k+1)}(x)+(k+1)f^{(k)}(x)+(k\theta+\theta)f^{(k+1)}(x),
\end{align*}
as required.
\end{proof}

\begin{equation*}\|f^{(2k+1)}\|\leq \sum_{j=0}^k\frac{(2k)!!}{(2j)!!}\frac{1}{(r+2j+2)_{k-j+1,2}}\bigg(\frac{2}{\sigma^2}\bigg)^{k-j+1}\|\tilde{h}^{(2j)}\|.
\end{equation*}

\begin{proof}Apply Lemma 2.3 with $a_k=k+1$ and $D_k=\frac{2}{\sigma^2(r+k)}$ to get
\begin{align*}\|f^{(2k+1)}\|&\leq \frac{2}{\sigma^2(r+k)}\sum_{j=0}^k\bigg(\prod_{i=j}^{k-1}\frac{2}{\sigma^2(r+2i)}\cdot(2i+2)\bigg)\|\tilde{h}^{(2j)}\|\\
&=\sum_{j=0}^k\frac{(2k)!!}{(2j)!!}\frac{1}{(r+2j+2)_{k-j+1,2}}\bigg(\frac{2}{\sigma^2}\bigg)^{k-j+1}\|\tilde{h}^{(2j)}\|.
\end{align*}
\end{proof}

\begin{align*}\|f^{(2k)}\|&\leq \sum_{j=1}^k\frac{(2k-1)!!}{(2j-1)!!}\frac{1}{(r+2j+1)_{k-j+1,2}}\bigg(\frac{2}{\sigma^2}\bigg)^{k-j+1}\|h^{(2j-1)}\|\\
&\quad+\frac{(2k-1)!!2^k}{(r+1)_{k,2}\sigma^{2k+1}}\bigg(\frac{1}{r}+\frac{\pi\Gamma(\frac{r}{2})}{2\Gamma(\frac{r+1}{2})}\bigg)\|\tilde{h}\|.
\end{align*}
The proof is very similar to the one for $\|f^{(2k+1)}\|$.

\subsection{Density proportional to $e^{-x^4/12}$}

\begin{equation}\label{4111}
Lf(x)= f'(x)-\frac{x^3}{3}f(x)=\tilde{h}(x).
\end{equation}
We have
\begin{eqnarray}
Lf'(x)&=&h'(x)+x^2f(x)=:h_1(x),\label{cwi11} \\
Lf''(x)&=&h''(x)+2xf(x)+2x^2f'(x)=:h_2(x)\label{cwi21},
\end{eqnarray}
and, for $k\geq3$,
\begin{align}Lf^{(k)}(x)&=h_k(x):=h^{(k)}(x)+\frac{1}{3}k(k-1)(k-2)f^{(k-3)}(x)+k(k-1)xf^{(k-2)}(x)\nonumber\\
\label{cwik1}&\quad+kx^2f^{(k-1)}(x).
\end{align}

\begin{proof}Differentiating both sides of (\ref{4111}) gives
\begin{equation*}f''(x)-\frac{x^3}{3}f'(x)-x^2f(x)=h'(x),
\end{equation*}
which on rearranging yields (\ref{cwi11}).  Differentiating (\ref{cwi11}) gives that
\begin{equation*}f^{(3)}(x)-\frac{x^3}{3}f''(x)-x^2f'(x)=h''(x)+2xf(x)+x^2f'(x)+2xf(x),
\end{equation*}
which on rearranging yields (\ref{cwi21}).  We now use induction to prove (\ref{cwik1}) for $k\geq3$.  Differentiating both sides of (\ref{cwik1}) gives
\begin{align*}Lf^{(k+1)}&=x^2f^{(k)}(x)+h^{(k+1)}(x)+2kxf^{(k-1)}(x)+kx^2f^{(k)}(x)+k(k-1)f^{(k-2)}(x)\\
&\quad+k(k-1)xf^{(k-1)}(x)+\frac{1}{3}k(k-1)(k-2)f^{(k-2)} \\
&=h^{(k+1)}(x)+(k+1)x^2f^{(k)}(x)+[2k+k(k-1)]xf^{(k-1)}(x)+[k(k-1)\\
&\quad+\tfrac{1}{3}k(k-1)(k-2)]f^{(k-2)}(x) \\
&=h^{(k+1)}(x)+(k+1)x^2f^{(k)}(x)+k(k+1)xf^{(k-1)}(x)+[k(k-1)\\
&\quad+\frac{1}{3}(k+1)k(k-1)]f^{(k-2)}(x)=h_{k+1}(x),
\end{align*}
as required.  
\end{proof}

\end{document}